\documentclass[12pt]{amsart}
\usepackage{amsmath}
\usepackage{amsfonts}
\usepackage{amssymb}
\allowdisplaybreaks 
\usepackage{mathtools}

\usepackage[all, cmtip]{xy} 
\usepackage{bbm}
\usepackage{bbding}
\usepackage{txfonts}
\usepackage[shortlabels]{enumitem}
\usepackage{ifpdf}
\ifpdf
\usepackage[colorlinks, final, backref=page, hyperindex]{hyperref}
\else
\usepackage[colorlinks, final, backref=page, hyperindex, hypertex]{hyperref}
\fi
\usepackage{mathrsfs}
\usepackage{amscd}
\usepackage{tikz-cd}
\usepackage[active]{srcltx}
\usepackage{tikz}
\usepackage{ulem}
\usetikzlibrary{calc}
\usetikzlibrary{arrows, shapes, chains}
\usetikzlibrary{arrows.meta}
\usetikzlibrary{decorations.markings}

\makeatletter

\theoremstyle{definition}
\newtheorem{thm}{Theorem}[section]
\newtheorem{prop}[thm]{Proposition}
\newtheorem{lem}[thm]{Lemma}
\newtheorem{cor}[thm]{Corollary}
\newtheorem{prop-def}[thm]{Proposition-Definition}
\newtheorem{defn}[thm]{Definition}

\newtheorem{remark}[thm]{Remark}

\newcommand{\nc}{\newcommand}

\nc{\delete}[1]{{}}
\nc{\mmargin}[1]{}
\nc{\Alg}{\mathrm{Alg}}
\nc{\NjO}{\mathrm{NjO}}
\nc{\NjA}{\mathrm{NjA}}
\nc{\NjL}{\mathrm{NjL}}
\nc{\FN}{\mathrm{FN}}
\nc{\RN}{\mathrm{RN}}
\nc{\G}{\mathrm{G}}
\nc{\rmH}{\mathrm{H}}
\nc{\DT}{\mathrm{DT}}
\nc{\ac}{\mathrm{\textup{!`}}}
	\nc{\mlabel}[1]{\label{#1}} 
	\nc{\mcite}[1]{\cite{#1}} 
	\nc{\mref}[1]{\ref{#1}} 
	\nc{\mbibitem}[1]{\bibitem{#1}} 

\delete{
	\nc{\mlabel}[1]{\label{#1} 
		{\hfill \hspace{1cm}{\bf{{\ }\hfill(#1)}}}}
	\nc{\mcite}[1]{\cite{#1}{{\bf{{\ }(#1)}}}} 
	\nc{\mref}[1]{\ref{#1}{{\bf{{\ }(#1)}}}} 
	\nc{\mbibitem}[1]{\bibitem[\bf #1]{#1}} 
	
}

 \font\cyrs=wncyr7

\nc{\vep}{\varepsilon}
\nc{\bin}[2]{ (_{\stackrel{\scs{#1}}{\scs{#2}}})} 
\nc{\binc}[2]{(\!\! \begin{array}{c} \scs{#1}\\
		\scs{#2} \end{array}\!\!)} 
\nc{\bincc}[2]{ ( {\scs{#1} \atop
		\vspace{-1cm}\scs{#2}} )} 
\nc{\oline}[1]{\overline{#1}}
\nc{\mapm}[1]{\lfloor\!|{#1}|\!\rfloor}
\nc{\bs}{\bar{S}}
\nc{\la}{\longrightarrow}
\nc{\ot}{\otimes}
\nc{\rar}{\rightarrow}
\nc{\lon }{\, \rightarrow\, }
\nc{\dar}{\downarrow}
\nc{\dap}[1]{\downarrow \rlap{$\scriptstyle{#1}$}}
\nc{\defeq}{\stackrel{\rm def}{=}}
\nc{\dis}[1]{\displaystyle{#1}}
\nc{\dotcup}{\ \displaystyle{\bigcup^\bullet}\ }
\nc{\hcm}{\ \hat{, }\ }
\nc{\hts}{\hat{\otimes}}
\nc{\hcirc}{\hat{\circ}}
\nc{\lleft}{[}
\nc{\lright}{]}
\nc{\curlyl}{\left \{ \begin{array}{c} {} \\ {} \end{array}
	\right. \!\!\!\!\!\!\!}
\nc{\curlyr}{ \!\!\!\!\!\!\!
	\left. \begin{array}{c} {} \\ {} \end{array}
	\right \} }
\nc{\longmid}{\left | \begin{array}{c} {} \\ {} \end{array}
	\right. \!\!\!\!\!\!\!}
\nc{\ora}[1]{\stackrel{#1}{\rar}}
\nc{\ola}[1]{\stackrel{#1}{\la}}
\nc{\scs}[1]{\scriptstyle{#1}} \nc{\mrm}[1]{{\rm #1}}
\nc{\dirlim}{\displaystyle{\lim_{\longrightarrow}}\, }
\nc{\invlim}{\displaystyle{\lim_{\longleftarrow}}\, }
\nc{\dislim}[1]{\displaystyle{\lim_{#1}}} \nc{\colim}{\mrm{colim}}
\nc{\mvp}{\vspace{0.3cm}} \nc{\tk}{^{(k)}} \nc{\tp}{^\prime}
\nc{\ttp}{^{\prime\prime}} \nc{\svp}{\vspace{2cm}}
\nc{\vp}{\vspace{8cm}}
\nc{\modg}[1]{\!<\!\!{#1}\!\!>}
\nc{\intg}[1]{F_C(#1)}
\nc{\lmodg}{\!<\!\!}
\nc{\rmodg}{\!\!>\!}
\nc{\cpi}{\widehat{\Pi}}
\nc{\ssha}{{\mbox{\cyrs X}}} 
\nc{\tsha}{{\mbox{\cyrt X}}}
\nc{\shpr}{\diamond} 
\nc{\labs}{\mid\!}
\nc{\rabs}{\!\mid}

\nc{\RBO}{{\mathrm{RBO}_\lambda}}
\nc{\Sh}{\mathrm{Sh}}
\nc{\sh}{\mathrm{\overline{Sh}}}
\nc{\RBA}{{\mathrm{RBA}_\lambda}}
\nc{\sgn}{\mathrm{sgn}}
\nc{\rH}{\mathrm{H}}

\nc{\ad}{\mrm{ad}}
\nc{\ann}{\mrm{ann}}
\nc{\Aut}{\mrm{Aut}}
\nc{\bbZ}{\mathbb{Z}}
\nc{\bim}{\mbox{-}\mathsf{Bimod}}
\nc{\br}{\mrm{bre}}
\nc{\can}{\mrm{can}}
\nc{\Cont}{\mrm{Cont}}
\nc{\rchar}{\mrm{char}}
\nc{\cok}{\mrm{coker}}
\nc{\de}{\mrm{dep}}
\nc{\dtf}{{R-{\rm tf}}}
\nc{\dtor}{{R-{\rm tor}}}
\nc{\Div}{{\mrm Div}}
\nc{\Diff}{\mrm{DA}}
\nc{\Diffl}{\mathsf{DA}_\lambda}
\nc{\diffo}{{\mathsf{DO}_\lambda}}
\nc{\alg}{\mathsf{Alg}}
\nc{\End}{\mrm{End}}
\nc{\Alt}{\mrm{Alt}}
\nc{\Ext}{\mrm{Ext}}
\nc{\Fil}{\mrm{Fil}}
\nc{\Fr}{\mrm{Fr}}
\nc{\Frob}{\mrm{Frob}}
\nc{\Gal}{\mrm{Gal}}
\nc{\GL}{\mrm{GL}}
\nc{\Hom}{\mrm{Hom}}
\nc{\Hoch}{\mrm{Hoch}}
\nc{\C}{\mrm{C}}
\nc{\hsr}{\mrm{H}}
\nc{\hpol}{\mrm{HP}}
\nc{\id}{\mrm{Id}}
\nc{\im}{\mrm{im}}
\nc{\Id}{\mrm{Id}}
\nc{\ID}{\mrm{ID}}
\nc{\Irr}{\mrm{Irr}}
\nc{\incl}{\mrm{incl}}
\nc{\length}{\mrm{length}}
\nc{\NLSW}{\mrm{NLSW}}
\nc{\Lie}{\mrm{Lie}}
\nc{\mchar}{\rm char}
\nc{\mpart}{\mrm{part}}
\nc{\ql}{{\QQ_\ell}}
\nc{\qp}{{\QQ_p}}
\nc{\rank}{\mrm{rank}}
\nc{\rcot}{\mrm{cot}}
\nc{\rdef}{\mrm{def}}
\nc{\rdiv}{{\rm div}}
\nc{\rtf}{{\rm tf}}
\nc{\rtor}{{\rm tor}}
\nc{\res}{\mrm{res}}
\nc{\SL}{\mrm{SL}}
\nc{\Spec}{\mrm{Spec}}
\nc{\tor}{\mrm{tor}}
\nc{\Tr}{\mrm{Tr}}
\nc{\tr}{\mrm{tr}}
\nc{\wt}{\mrm{wt}}
\def\ot{\otimes}

\nc{\bfk}{{\bf k}}
\nc{\bfone}{{\bf 1}}
\nc{\bfzero}{{\bf 0}}
\nc{\detail}{\marginpar{\bf More detail}
	\noindent{\bf Need more detail!}
	\svp}
\nc{\gap}{\marginpar{\bf Incomplete}\noindent{\bf Incomplete!!}
	\svp}
\nc{\FMod}{\mathbf{FMod}}
\nc{\Int}{\mathbf{Int}}
\nc{\Mon}{\mathbf{Mon}}
\nc{\sproof}{\noindent{ \textit{Sketch of Proof:} }}
\nc{\remarks}{\noindent{\bf Remarks: }}
\nc{\Rep}{\mathbf{Rep}}
\nc{\Rings}{\mathbf{Rings}}
\nc{\Sets}{\mathbf{Sets}}
\nc{\ob}{\mathsf{Ob}}

\nc{\BA}{{\mathbb A}} \nc{\CC}{{\mathbb C}}
\nc{\DD}{{\mathbb D}} \nc{\EE}{{\mathbb E}}
\nc{\FF}{{\mathbb F}} \nc{\GG}{{\mathbb G}}
\nc{\LL}{{\mathbb L}}
\nc{\NN}{{\mathbb N}} \nc{\PP}{{\mathbb P}}
\nc{\QQ}{{\mathbb Q}} \nc{\RR}{{\mathbb R}}
\nc{\TT}{{\mathbb T}} \nc{\VV}{{\mathbb V}}
\nc{\ZZ}{{\mathbb Z}} \nc{\TP}{\widetilde{P}}
\nc{\m}{{\mathbbm m}}

\nc{\cala}{{\mathcal A}} \nc{\calc}{{\mathcal C}}
\nc{\cald}{\mathcal{D}} \nc{\cale}{{\mathcal E}}
\nc{\calF}{{\mathcal F}} \nc{\calg}{{\mathcal G}}
\nc{\calh}{{\mathcal H}} \nc{\cali}{{\mathcal I}}
\nc{\call}{{\mathcal L}} \nc{\calm}{{\mathcal M}}
\nc{\caln}{{\mathcal N}} \nc{\calo}{{\mathcal O}}
\nc{\calp}{{\mathcal P}} \nc{\calr}{{\mathcal R}}
\nc{\cals}{{\mathcal S}} \nc{\calt}{{\Omega}}
\nc{\calv}{{\mathcal V}} \nc{\calw}{{\mathcal W}}
\nc{\calx}{{\mathcal X}}

\nc{\fraka}{{\mathfrak a}}
\nc{\frakb}{\mathfrak{b}}
\nc{\frakC}{\mathfrak{C}}
\nc{\frakF}{\mathfrak{F}}
\nc{\frakg}{{\frak g}}
\nc{\frakl}{{\frak l}}
\nc{\frakh}{{\frak h}}
\nc{\fraks}{{\frak s}}
\nc{\frakB}{{\frak B}}
\nc{\frakm}{{\frak m}}
\nc{\frakM}{{\frak M}}
\nc{\frakp}{{\frak p}}
\nc{\frakW}{{\frak W}}
\nc{\frakX}{{\frak X}}
\nc{\frakS}{{\frak S}}
\nc{\frakT}{{\frak T}}
\nc{\frakA}{{\frak A}}
\nc{\frakL}{{\frak L}}
\nc{\frakx}{{\frakx}}
\nc{\RB}{\mathfrak{RB}}
\nc{\NjAoperad}{\mathfrak{NjA}}
\nc{\NjAModoperad}{\mathfrak{NjA\hspace{-0.15cm}+\hspace{-0.15cm}Mod}}
\nc{\NjAMoroperad}{\mathfrak{NjA\hspace{-0.15cm}+\hspace{-0.15cm}Mor}}
\nc{\NL}{\mathfrak{NjL}}

\nc{\lbar}[1]{\overline{#1}}
\nc{\ra}{\rightarrow}
\nc{\tred}[1]{\textcolor{red}{#1}} \nc{\tgreen}[1]{\textcolor{green}{#1}}
\nc{\tblue}[1]{\textcolor{blue}{#1}} \nc{\tpurple}[1]{\textcolor{purple}{#1}}

\nc{\blue}{{\color{blue}{\mathrm{blue}}}}
\nc{\red}{{\color{red}{\mathrm{red}}}}

\nc{\black}{\mathrm{black}}

\begin{document}
	
\title[Homotopy Nijenhuis algebras]{Deformations and homotopy theory of Nijenhuis associative algebras}

\author{Chao Song, Kai Wang, Yuanyuan Zhang and Guodong Zhou}

\date{\today}

\address{Chao Song, School of Mathematical Sciences, East China Normal University, Shanghai 200241, China
}
\email{52265500011@stu.ecnu.edu.cn}

\address{Guodong Zhou, School of Mathematical Sciences, Key Laboratory of Mathematics and Engineering Applications (Ministry of Education), Shanghai Key Laboratory of PMMP, East China Normal University, Shanghai 200241, China
}
\email{gdzhou@math.ecnu.edu.cn}

\address{Kai Wang, School of Mathematical Sciences, University of Science and Technology of China, Hefei, Anhui Provience 230026, China}
	\email{wangkai17@ustc.edu.cn}

\address{Yuanyuan Zhang, School of Mathematics and Statistics, Henan University, Henan, Kaifeng 475004, China
}
\email{zhangyy17@henu.edu.cn}

\begin{abstract}

This paper is the first in a series of works devoted to an operadic study of Nijenhuis structures, focusing on Nijenhuis associative algebras.
We introduce the concept of homotopy Nijenhuis associative algebras and demonstrate that the differential graded (=dg) operad $\NjAoperad_{\infty}$ governing these structures serves as the minimal model of the operad $\NjAoperad$ for Nijenhuis associative algebras. Additionally, we determine the Koszul dual homotopy cooperad of $\NjAoperad$.
We construct an $L_\infty$-algebra that controls the simultaneous deformations of associative products and Nijenhuis operators. The Maurer-Cartan elements of this $L_\infty$-algebra correspond bijectively to Nijenhuis associative algebra structures. From this, we derive a cochain complex (deformation complex) and an associated cohomology theory of Nijenhuis associative algebras.
Finally, we explore the connection between homotopy relative Rota-Baxter associative algebras of weight $0$ and homotopy Nijenhuis associative algebras. A sequel to this work will extend the study to Nijenhuis Lie algebras, with applications to Nijenhuis geometry.

\end{abstract}

\subjclass[2020]{
	16E40 
	16S80 
	17B38 
	18M60 
	18M65 
	18M70 
}

\keywords{Cohomology, deformation complex, homotopy Nijenhuis associative algebra, Koszul dual homotopy cooperad, minimal model, Nijenhuis associative algebra, operad}

\maketitle

\tableofcontents

\allowdisplaybreaks
	
\section*{Introduction}

A Nijenhuis operator on a Lie algebra $(\frakg, [-, -])$ is a linear endomorphism $P$ satisfying the following Nijenhuis relation:
\begin{align} \label{Nijenhuis relation}
	[P(a), P(b)]= P( [a, P(b)] + [P(a), b] -P[a, b]), \text{ for } a, b \in \frakg.
\end{align}
The concept of a Nijenhuis operator on a Lie algebra originated from the foundational notion of a Nijenhuis tensor, introduced by Nijenhuis \cite{Nij51} in the study of pseudo-complex manifolds. Nijenhuis operators on Lie algebras were later explored in \cite{MM84, Kos90} within the broader context of Poisson-Nijenhuis manifolds. On manifolds, Nijenhuis operators act as unary operations on the Lie algebra of vector fields and satisfy the corresponding Nijenhuis relation \eqref{Nijenhuis relation}, which is the simplest geometric condition for $(1, 1)$-tensors, as outlined in the philosophy of \cite{BKM22}. In 1957, Newlander and Nirenberg \cite{NN57} proved that an almost complex manifold $(M, J)$ is a complex manifold if and only if $J$ is a Nijenhuis operator, as also noted in \cite[Theorem 1.2, Page356]{Hel78}. This provides a significant example of a Nijenhuis operator.


The concept of Nijenhuis operators on associative algebras, introduced by Carinena et al. \cite{CGM00} to study quantum bi-Hamiltonian systems, is an analogue of the Nijenhuis operators on Lie algebras and manifolds (see Definition~\ref{Def: Nijenhuis associative algebra}).
In \cite{EbrahimiFard2004}, Ebrahimi-Fard interpreted the associative analogue of the Nijenhuis relation as the homogeneous version of the Rota-Baxter relation. Building on this perspective, he used the augmented modified quasi-shuffle product to construct free Nijenhuis associative algebras. Guo and Lei \cite{LeiGuo} provided an explicit construction of free Nijenhuis associative algebras using bracketed words, and applied this construction to derive the universal enveloping Nijenhuis associative algebra of an NS algebra. An important example of the Nijenhuis operator on associative algebras is given by Das \cite[2.5 Proposition]{Das2020}, which establishes a connection between relative Rota-Baxter operators of weight 0 and Nijenhuis operators. And this connection will be generalized to homotopy level in this paper.

We focus on the deformation and homotopy theory of Nijenhuis structures. Specifically, in this paper, we will study the deformation and homotopy theory of Nijenhuis associative structures. In subsequent work, we will consider these problems for Nijenhuis structures on Lie algebras and manifolds, as well as their applications in geometry.

A central philosophy in deformation theory, inspired by the foundational work of Gerstenhaber, Nijenhuis, Richardson, Deligne, Schlessinger, Stasheff, Goldman, Millson, and others, is that the deformations of a given mathematical structure can be described by a differential graded (dg) Lie algebra or, more generally, an $L_\infty$-algebra. The underlying complex of this algebra is referred to as the deformation complex. Lurie \cite{Lurie} and Pridham \cite{Pri10} formalized this idea rigorously, establishing it as an equivalence of infinity categories in characteristic zero. Consequently, an essential problem in deformation theory is the explicit construction of the dg Lie algebra or $L_\infty$-algebra that governs the deformation theory of the structure under study.

Another significant challenge in the study of algebraic structures is understanding their homotopy analogs, such as $A_\infty$-algebras for associative algebras and $L_\infty$-algebras for Lie algebras. Operad theory provides a systematic framework for this task, requiring the formulation of a cofibrant resolution of the operad describing the algebraic structure. Ideally, this resolution takes the form of a minimal model of the operad. If the operad is Koszul, Koszul duality theory \cite{GJ94, GK94} offers a powerful method to define the homotopy version of the structure through the cobar construction of the Koszul dual cooperad, which in this case is itself a minimal model.
However, for operads that are not Koszul, significant challenges emerge, and examples of minimal models in this setting remain scarce. This paper contributes to addressing these difficulties in the context of Nijenhuis structures.

These two problems, say, describing controlling $L_\infty$-algebras and constructing homotopy versions of algebraic structures, are deeply interconnected. Specifically, given a cofibrant resolution, particularly a minimal model of the operad in question, one can derive the deformation complex of the algebraic structure and endow it with an $L_\infty$-structure, as demonstrated by Kontsevich and Soibelman \cite{KS00} and van der Laan \cite{VdLa, VdLb}.
In practice, however, finding a minimal model or a small cofibrant resolution is often a nontrivial challenge. Recently, Wang and Zhou \cite{WZ24} successfully addressed both of these problems for Rota-Baxter associative algebras of arbitrary weight, while Chen, Guo, Wang, and Zhou \cite{CGWZ24} extended these results to differential algebras with nonzero weight.
In this paper and its sequel, we extend these methods, with significant modifications, to study the operad governing Nijenhuis structures.



In this paper, we exhibit the deformation complex of a Nijenhuis associative algebra as well as the controlling $L_\infty$-structure in the sense of deformation theory. We succeed in finding the minimal model and the Koszul dual homotopy cooperad of the operad for Nijenhuis associative algebras, providing justification for the deformation complex as well as the controlling $L_\infty$-structure.

This paper is organized as follows.
In Section~\ref{Nijenhuis associative algebras and Nijenhuis bimodules}, we recall basic definitions and key facts about Nijenhuis associative algebras. In Section~\ref{Sect: Cohomology theory of Nijenhuis associative algebras}, we introduce a cochain complex of Nijenhuis operators, whose homology controls the deformation of Nijenhuis operators. We then combine this cochain complex with the usual Hochschild cochain complex of associative algebras to define a cochain complex for Nijenhuis associative algebras, whose homology controls the deformations of Nijenhuis associative algebras. Section~\ref{Section: Koszul dual cooperad NjA} introduces a homotopy cooperad, and in Section~\ref{Section: minimal model NjA}, we show that the cobar construction of this homotopy cooperad, using a new monomial order and the method in \cite{WZ24}, is the minimal model of the operad for Nijenhuis associative algebras. As a result, this homotopy cooperad can be viewed as the Koszul dual of the operad for Nijenhuis associative algebras. The concept of homotopy Nijenhuis associative algebras is also introduced. In Section~\ref{Section: From minimal model to Linifnity algebras NjA}, we construct the $L_\infty$-algebra structure on the deformation complex introduced in Section~\ref{Sect: Cohomology theory of Nijenhuis associative algebras}, and Nijenhuis associative algebra structures are realized as the Maurer-Cartan elements of this $L_\infty$-algebra. Finally, in Section~\ref{Section: relation between HrelRBA and HNjA}, we establish a relationship between homotopy Rota-Baxter associative algebras of weight $0$ and homotopy Nijenhuis associative algebras, which generalizes a result of Das \cite{Das2020} to homotopy level.

Throughout this paper, let $\bfk$ be a field of characteristic $0$. Unless otherwise specified, all vector spaces are assumed to be $\bfk$-vector spaces, and all tensor products and $\Hom$-spaces are taken over $\bfk$. Additionally, all algebras are defined over $\bfk$.
For foundational concepts on operads and colored operads, we refer the reader to \cite{MSS02, Mar08, LV12, BD16, Yau16}. Familiarity with the basic theory of $L_\infty$-algebras \cite{Sta92, LS93, LM95} and homotopy (co)operads \cite{Mar96, MV09a, MV09b, DP16} is also assumed. For a concise overview, we recommend \cite[Sections 2.2 and 3.1]{WZ24}.
 
\medskip

\section{Nijenhuis associative algebras and Nijenhuis bimodules} \label{Nijenhuis associative algebras and Nijenhuis bimodules}

In this section, we recall some basic concepts and facts about Nijenhuis associative algebras and their Nijenhuis bimodules.

\begin{defn} \label{Def: Nijenhuis associative algebra}
	Let $(A, m=\cdot)$ be an associative algebra. A linear operator $P: A \rightarrow A$ is a \textit{Nijenhuis operator} if it satisfies
	\begin{eqnarray*} \label{Eq: Nijenhuis relation}
		P(a) \cdot P(b) = P \big(P(a) \cdot b + a \cdot P(b) - P( a \cdot b)\big)
	\end{eqnarray*}	
	for any $a, b \in A$, or in terms of maps
	\begin{eqnarray} \label{Eq: Nijenhuis relation in terms of maps}
		m \circ (P \ot P)=P\circ \big(m \circ (P \ot \Id) + m \circ (\Id\ot P) - P\circ m \big).
	\end{eqnarray}
	In this case, $A = (A, m, P)$ is called a \textit{Nijenhuis associative algebra}.

	Let $B=(B, m_{B}, P_{B})$ and $A=(A, m_A, P_{A})$ be two Nijenhuis associative algebras. A morphism of associative algebras $f: B \rightarrow A$ is called a \textit{morphism of Nijenhuis associative algebras}, if $ P_{A} \circ f = f \circ P_{B}$.

\end{defn}

\begin{defn} \label{Def: Nijenhuis bimodules}
	Let $A = (A, m, P)$ be a Nijenhuis associative algebra and $M$ be a bimodule over the associative algebra $(A, m)$.
	Then $M$ is said to be a \textit{Nijenhuis bimodule} over Nijenhuis associative algebra $A$, if $M$ is endowed with a linear operator $P_M: M \rightarrow M$ such that the following equations
	\begin{align} \label{eq:left action}
		P(a) P_M(x) &= P_M \big(P(a)x + aP_M(x) - P_M (ax)\big), \end{align}
	\begin{align} \label{eq:right action}	
		P_M(x) P(a) &= P_M \big(P_M(x)a + xP(a) - P_M (xa)\big),
	\end{align}
	hold for any $a\in A$ and $x\in M$.
\end{defn}

Obviously, a Nijenhuis associative algebra $A = (A, m, P)$ is a Nijenhuis bimodule over itself, called the \textit{regular Nijenhuis bimodule}. 

Let $(A, m)$ be an associative algebra and $M$ be a bimodule over it. It is well known that $A\oplus M$ becomes an associative algebra whose product is given by
	\begin{eqnarray*}
		(a, x) \cdot_{\ltimes} (b, y) := (a\cdot b, ay+xb),
	\end{eqnarray*}
	for any $a, b \in A$ and $x, y \in M$. This is called the \textit{semi-direct product (or trivial extension)} of associative algebra $(A, m)$ by bimodule $M$, denoted by $A\ltimes M$.
 Write $\iota: A \rightarrow A\oplus M, a\mapsto (a, 0)$ and $\pi: A\oplus M \rightarrow A, (a, x)\mapsto a$, which are obviously associative algebra morphisms.

 The following results are direct consequences of the definitions, whose proofs are left to the reader.
\begin{prop} \label{Prop: trivial extension of Nijenhuis bimodule}
	Let $A = (A, m, P)$ be a Nijenhuis associative algebra and $M$ be a bimodule over associative algebra $(A, m)$. Then the semi-direct product $A\ltimes M$ of $A$ by $M$ is a Nijenhuis associative algebra such that $\iota$ and $\pi$ are both morphisms of Nijenhuis associative algebras if and only if $M$ is a Nijenhuis bimodule over $A$.
	
	This new Nijenhuis associative algebra will still be denoted by $A\ltimes M$, called the \textit{(Nijenhuis) semi-direct product (or trivial extension)} of Nijenhuis associative algebra $A$ by Nijenhuis bimodule $M$.
\end{prop}

Nijenhuis associative algebras and Nijenhuis bimodules have some descendent properties.

\begin{prop}\cite{CGM00} 
	\label{Prop: new Nj algebra}
	Let $A = (A, m, P)$ be a Nijenhuis associative algebra. Define a new binary operation $m_P=\cdot_P$ over $A$ as:
	\begin{eqnarray*}
		a \cdot_P b := P(a)\cdot b + a\cdot P(b) - P(a\cdot b),
	\end{eqnarray*}
	for any $a, b\in A$. Then
	\begin{enumerate}
		\item the operation $m_P $ is associative and $(A, m_P)$ is a new associative algebra;
		\item the triple $(A, m_P, P)$ also forms a Nijenhuis associative algebra and denote it by $A_P $;
		\item the map $P: A_P=(A, m_P, P)\rightarrow A = (A, m, P)$ is a morphism of Nijenhuis associative algebras.
	\end{enumerate}
\end{prop}

One can also obtain new Nijenhuis bimodules from old ones.

\begin{prop}[{Compare with \cite[2.8 Example]{Das2020 twisted}}]\label{Prop: new-bimodule}
	Let $A = (A, m, P)$ be a Nijenhuis associative algebra and $M = (M, P_M)$ be a Nijenhuis bimodule over it. We define a left action $``\rhd"$ and a right action $``\lhd"$ of $A$ on $M$ as follows:
	\begin{align*}
		a\rhd x :=P(a)x \quad \mathrm{and} \quad x\lhd a :=xP(a)
	\end{align*}
	for any $a\in A$ and $ x\in M$.
	Then these actions make $M$ into a Nijenhuis bimodule over $A_P $, and we denote this new bimodule by $M_P$.
\end{prop}

Let's describe the operad for Nijenhuis associative algebras.
\begin{defn} \label{def: operad NjA}
	The \textit{(nonsymmetric) operad for Nijenhuis associative algebras}, denoted by $\NjAoperad$, is generated by a unary operator $P$ and a binary operator $m$ with the operadic relations:
	\begin{align}\label{Eq: operadic associativity}
		m \circ_1 m - m \circ_2 m
	\end{align}
	and
	\begin{align}\label{Eq: operadic Nijenhuis relation}
		(m \circ_1 P)\circ_2 P-(P \circ_1 m)\circ_1 P-(P \circ_1 m)\circ_2 P + P \circ_1 (P \circ_1 m).
	\end{align}

In other words, $\NjAoperad=\mathcal{F}(\bfk m\oplus \bfk P)/I,$ where for a collection $M$, $\mathcal{F}(M)$ is the free operad generated by $M$, and $I$ is the operadic ideal generated by \eqref{Eq: operadic associativity} and \eqref{Eq: operadic Nijenhuis relation}.

Here, Relation~\eqref{Eq: operadic associativity} is the associativity axiom for the product $m$ and Relation~\eqref{Eq: operadic Nijenhuis relation} corresponds to the defining relation \eqref{Eq: Nijenhuis relation in terms of maps} of the Nijenhuis operator $P$.
\end{defn}

\begin{remark}
	Since the defining relation \eqref{Eq: operadic Nijenhuis relation} of the Nijenhuis operator is cubic, it is obvious that the operad $\NjAoperad$ is not Koszul.
	So the classical Koszul duality theory of operads \cite{GJ94, GK94} could not be applied directly to develop the operadic cohomology theory, minimal model, controlling $L_\infty$-structures, etc., for Nijenhuis associative algebras.
\end{remark}

\medskip

\section{Cohomology theory of Nijenhuis associative algebras} \label{Sect: Cohomology theory of Nijenhuis associative algebras}

In this section, we will introduce a cohomology theory of Nijenhuis associative algebras. We will see later that this cohomology theory controls the deformations of Nijenhuis associative algebras.

\smallskip

\subsection{Hochschild cohomology of associative algebras} \label{Hochschild cohomology of associative algebras}\


Let $M$ be a bimodule over an associative algebra $(A, m)$. Recall that the \textit{Hochschild cochain complex of associative algebra $(A, m)$ with coefficients in bimodule $M$} is the cochain complex
$$\C^\bullet_\Alg(A, M):= \bigoplus_{n=0}^\infty \C^n_\Alg(A, M ),$$
where, for $n \geqslant 0$, $\C^n_\Alg(A, M )=\Hom(A^{\otimes n}, M )$ and its differential 
$$\delta^n_{\Alg, M}: \C^n_\Alg(A, M ) \rightarrow \C^{n+1}_\Alg(A, M )$$
is given by
\begin{align*} 
	\delta_{\Alg, M}^n(f)(a_{1, n+1})= a_1 f(a_{2, n+1})+\sum_{i=1}^n(-1)^{i}f(a_{1, i-1} \ot a_i \cdot a_{i+1}\ot a_{i+2, n+1})
	+ (-1)^{n+1}f(a_{1, n}) a_{n+1}
\end{align*}
for any $f\in \C^n_\Alg(A, M )$ and $a_1, \dots, a_{n+1}\in A$, where for $1 \leqslant i \leqslant j \leqslant n$, $a_{i, j}=a_i\otimes \cdots \otimes a_j$ but for $i>j$, $a_{i, j}$ is by convention empty.

The cohomology of the cochain complex, denoted by $\rmH^\bullet_\Alg(A, M )$, is called the \textit{Hochschild cohomology of associative algebra $(A, m)$ with coefficients in bimodule $M$}.
When $M$ is the regular bimodule $A$ itself, we simply write the differential $\delta^{\bullet}_{\Alg, A}$ as $\delta^{\bullet}_{\Alg}$, and also simply denote $\C^\bullet_\Alg(A, A)$ by $\C^\bullet_\Alg(A)$ and $\rmH^\bullet_\Alg(A, A)$ by $\rmH^\bullet_\Alg(A)$, called the \textit{Hochschild cochain complex of associative algebra $(A, m)$} and the \textit{Hochschild cohomology of associative algebra $(A, m)$}, respectively.


\smallskip

\subsection{Cohomology of Nijenhuis operators} \label{Subsect: cohomology NjA operator} \

Next, let's introduce the cohomology theory of Nijenhuis operators.

Let $A = (A, m, P)$ be a Nijenhuis associative algebra and $M=(M, P_M)$ be a Nijenhuis bimodule over it. According to Propositions \ref{Prop: new Nj algebra} and \ref{Prop: new-bimodule}, one can define a new Nijenhuis associative algebra $A_P=(A, m_P, P) $ and a new Nijenhuis bimodule $M_P=(M, \rhd, \lhd, P_M)$ over $A_P $.
Consider the Hochschild cochain complex of the associative algebra $A_P$ with coefficients in bimodule $M_P$:
$$\C^\bullet_{\mathrm{Alg}}(A_P, {M_P})=\bigoplus\limits_{n=0}^\infty \C^n_{\mathrm{Alg}}(A_P, {M_P}),$$
where, for $n \geqslant 0$, $ \C^n_{\mathrm{Alg}}(A_P, {M_P})=\Hom (A^{\ot n}, M)$ and its differential
$$\partial^n: \C^n_{\mathrm{Alg}}(A_P, {M_P})\rightarrow \C^{n+1}_{\mathrm{Alg}}(A_P, {M_P}) $$
is given by
\begin{align*}
	\partial^n(f)(a_{1, n+1})
	=&\ a_1\rhd f(a_{2, n+1})+\sum_{i=1}^n (-1)^{i} f (a_{1, i-1}\ot a_{i}\cdot_P a_{i+1} \ot a_{i+2, n+1})
	+ (-1)^{n+1} f(a_{1, n})\lhd a_{n+1}\\
	=&\ P(a_1) f(a_{2, n+1})\\
	& +\sum_{i=1}^n(-1)^{i} \big( f(a_{1, i-1}\ot a_{i} P(a_{i+1}) \ot a_{i+2, n+1})+f(a_{1, i-1}\ot P(a_{i})a_{i+1} \ot a_{i+2, n+1})\\
	& \quad \quad - f(a_{1, i-1}\ot P(a_{i} a_{i+1}) \ot a_{i+2, n+1}) \big) \\
	& + (-1)^{n+1} f(a_{1, n}) P(a_{n+1})
\end{align*}
for any $f\in \C^n_{\Alg}(A_P, M_P)$ and $a_1, \dots, a_{n+1}\in A$.

\begin{defn}
	Let $A = (A, m, P)$ be a Nijenhuis associative algebra and $M = (M, P_M)$ be a Nijenhuis bimodule over it. Define the \textit{cochain complex of Nijenhuis operator $P$ with coefficients in Nijenhuis bimodule $M$}, denoted by $(\C_{\NjO}^\bullet(A, M), \delta_{\NjO, M}^\bullet)$, as follows:
	for any $n\geqslant 0 $,
	$$\C_{\NjO}^n(A, M):=\Hom(A^{\ot n}, M)$$
	and its differential
	$$\delta_{\NjO, M}^n: \C^n_{\NjO}(A, M )\rightarrow \C^{n+1}_{\NjO}(A, M) $$
	is defined to be
	\begin{align} \label{Eq: diff of Nijenhuis operator}
		\delta^n_{\NjO, M}(f) := - P_M\circ \delta_{\Alg, M}^{n}(f)+\partial^{n}(f)
	\end{align}
	for any $f\in \C_{\NjO}^n(A, M)$.
		
	The cohomology of $\C_{\NjO}^\bullet(A, M)$, denoted by $\mathrm{H}_{\NjO}^\bullet(A, M)$, is called the\textit{ cohomology of Nijenhuis operator $P$ with coefficients in Nijenhuis bimodule $M$}.	
	When $M$ is the regular Nijenhuis bimodule $A$ itself,
	we simply denote the differential $\delta^{\bullet}_{\NjO, A}$ by $\delta^{\bullet}_{\NjO}$. We also simply denote $\C^\bullet_{\NjO}(A, A)$ by $\C^\bullet_{\NjO}(A)$ and $\rmH^\bullet_{\NjO}(A, A)$ by $\rmH^\bullet_{\NjO}(A)$, called the \textit{cochain complex of Nijenhuis operator $P$} and the \textit{cohomology of Nijenhuis operator $P$}, respectively.
\end{defn}

The above definition is justified by the following lemma, which can be verified through direct inspection. However, we will first deduce the specific case where \( M \) is the regular Nijenhuis bimodule \( A \) itself using a particular \( L_\infty \)-algebra (see Proposition~\ref{Prop: justifying Cohomology theory of Nijenhuis operator Ass} (ii)). Building on this result, the general cases will naturally follow as a direct corollary from the results presented in Subsection~\ref{Sub: cohomology proofs}.

\begin{lem} \label{Lem: cochain complex of Nij operator}
	$(\C_{\NjO}^\bullet(A, M), \delta_{\NjO, M}^\bullet)$ is a cochain complex.
\end{lem}

\begin{remark} The differential $\delta_{\NjO, M}^\bullet$ on the cochain complex of Nijenhuis operator $P$ with coefficients in Nijenhuis bimodule $M$ uses not only the Hochschild differential $\partial$ of the new associative algebra $A_P $ with coefficients in the new bimodule $M_P$, but also an extra term $-P_M\circ \delta_{\Alg, M}$. This distinctive feature, which has yet to be fully understood, is quite different from the known cases of other operated algebras.
In fact, the cohomology of a Rota-Baxter operator is defined to be the Hochschild cohomology of the new associative algebra with coefficients in the new bimodule (see \cite[Definition 6.1]{WZ24}) and the cohomology of a differential operator is defined to be the Hochschild cohomology of the new associative algebra with coefficients in the new bimodule (see \cite[2.8 Definition]{GLSZ22} or \cite[Section 2.1]{CGWZ24}).
\end{remark}

\smallskip

\subsection{Cohomology of Nijenhuis associative algebras} \label{Subsec: cohomology NjA} \

Now, we define a cohomology theory of Nijenhuis associative algebras by combining the Hochschild cohomology of associative algebras and the cohomology of Nijenhuis operators.

Let $M=(M, P_M)$ be a Nijenhuis bimodule over a Nijenhuis associative algebra $A = (A, m, P)$. Let's construct a chain map
$$\Phi_{M}^\bullet:\C^\bullet_{\Alg}(A, M) \rightarrow \C_{\NjO}^\bullet(A, M), $$
i.e., the following commutative diagram:
$$\xymatrix{
	\C^0_{\Alg}(A, M)\ar[r]^-{\delta_{\Alg, M}^0}\ar[d]^-{\Phi_{M}^0}& \C^1_{\Alg}(A, M)\ar[r]\ar[d]^-{\Phi_{M}^1} &\cdots \ar[r] &\C^n_{\Alg}(A, M)\ar[r]^-{\delta_{\Alg, M}^n}\ar[d]^-{\Phi_{M}^n}&\C^{n+1}_{\Alg}(A, M)\ar[d]^{\Phi_{M}^{n+1}}\ar[r] & \cdots\\
	\C^0_{\NjO}(A, M)\ar[r]^-{\delta^0_{\NjO, M}}&\C^1_{\NjO}(A, M)\ar[r] &\cdots \ar[r]& \C^n_{\NjO}(A, M)\ar[r]^-{\delta^n_{\NjO, M}}&\C^{n+1}_{\NjO}(A, M)\ar[r]&\cdots
	,}$$
where $\Phi_{M}^0 := \Id_{M}$, and for $n\geqslant 1$, $ f\in \C^n_{\Alg}(A, M)$, $\Phi_{M}^n(f)\in \C^n_{\NjO}(A, M)$ is defined as:
\begin{align*}
	\Phi_{M}^n(f)&(a_1, \dots, a_n)\\
	& := \sum_{k=0}^n \sum_{1\leqslant i_1<i_2<\dots<i_k\leqslant n} (-1)^{n-k}
	P_M^{n-k} \circ f(a_{1, i_1-1}, P(a_{i_1}), a_{i_1+1, i_2-1}, P(a_{i_2}), \dots, P(a_{i_k}), a_{i_k+1, n}).
\end{align*}
In particular, when $M$ is the regular Nijenhuis bimodule $A$ itself, we simply denote $\Phi^\bullet_{M}$ by $\Phi^\bullet$.

\begin{prop} \label{Prop: Chain map Phi}
	The map $\Phi_{M}^\bullet: \C^\bullet_\Alg(A, M)\rightarrow \C^\bullet_{\NjO}(A, M)$ is a chain map.
\end{prop}

The proof of Proposition~\ref{Prop: Chain map Phi} can be done by direct inspection, which is not difficult. But we will deduce Proposition~\ref{Prop: Chain map Phi} in the particular case where $M$ is the regular Nijenhuis bimodule $A$ itself, from a specific $L_\infty$-algebra, see Proposition~\ref{Prop: cohomology complex as the underlying complex of L-infinity algebra NjA}. On this basis, the general cases will be a direct corollary of the results in Subsection~\ref{Sub: cohomology proofs}.

\begin{defn} \label{Def: definition of Nijenhuis cohomology}
	Let $M=(M, P_M)$ be a Nijenhuis bimodule over a Nijenhuis associative algebra $A=(A, m, P)$. We define the \textit{cochain complex of Nijenhuis associative algebra $A$ with coefficients in Nijenhuis bimodule $M$}, denoted by $(\C^\bullet_{\NjA}(A, M), \delta_{\NjA, M}^\bullet)$, to be the mapping cone of $\Phi_{M}^\bullet$ shifted by $-1$, that is,
	$$\C^0_{\NjA}(A, M)=\C^0_\Alg(A, M) \quad \mathrm{and} \quad \C^n_{\NjA}(A, M)=\C^n_\Alg(A, M)\oplus \C^{n-1}_{\NjO}(A, M), \text{ for } n\geqslant 1, $$
	and its differential $\delta_{\NjA, M}^n: \C^n_{\NjA}(A, M)\rightarrow \C^{n+1}_{\NjA}(A, M)$ is given by
	\begin{align*}
		\delta_{\NjA, M}^n(f, g)
		=&\ (\delta_{\Alg, M}^n(f), -\Phi_{M}^n(f)-\delta_{\NjO, M}^{n-1}(g))
	\end{align*}	
	for any $f\in \C^n_\Alg(A, M)$ and $g\in \C^{n-1}_{\NjO}(A, M)$.

	The cohomology of $\C^\bullet_{\NjA}(A, M)$, denoted by $\rmH_{\NjA}^\bullet(A, M)$, is called the \textit{cohomology of Nijenhuis associative algebra $A$ with coefficients in Nijenhuis bimodule $M$}.
	When $M$ is the regular Nijenhuis bimodule $A$ itself,
	we simply denote the differential $\delta^{\bullet}_{\NjA, A}$ by $\delta^{\bullet}_{\NjA}$. We also simply denote $\C^\bullet_{\NjA}(A, A)$ by $\C^\bullet_{\NjA}(A)$ and $\rmH^\bullet_{\NjA}(A, A)$ by $\rmH^\bullet_{\NjA}(A)$, called the \textit{cochain complex of Nijenhuis associative algebra $A$} and the \textit{cohomology of Nijenhuis associative algebra $A$}, respectively.
\end{defn}

By properties of the mapping cone, there is a short exact sequence of cochain complexes:
\begin{eqnarray*} 
	0 \to s^{-1}\C^\bullet_{\NjO}(A, M) \to \C^\bullet_{\NjA}(A, M) \to \C^\bullet_{\Alg}(A, M)\to 0,
\end{eqnarray*}
which induces a long exact sequence of cohomology groups
$$0 \to \rmH^{0}_{\NjA}(A, M) \to \mathrm{H}_{\Alg}^0(A, M) \to \rmH^0_{\NjO}(A, M) \to \rmH^{1}_{\NjA}(A, M) \to \mathrm{H}_{\Alg}^1(A, M) \to \cdots$$
$$\cdots \to \mathrm{H}_{\Alg}^p(A, M) \to \rmH^p_{\NjO}(A, M) \to \rmH^{p+1}_{\NjA}(A, M) \to \mathrm{H}_{\Alg}^{p+1}(A, M) \to \cdots$$
Here the symbol $s^{-1}$ represents the desuspension of a cohomologically graded vector space. For a homologically graded space $V$, the suspension of $V$ is the graded space $sV $ with $(sV)_n=V_{n-1}$ for all $n\in \bbZ$. Write $sv\in (sV)_n$ for $v\in V_{n-1}$. The map $s: V \rightarrow sV$, $v\mapsto sv$ is a graded map of degree $1$.
 Similarly, the desuspension of $V$ , denoted $s^{-1}V$, is the graded space with $(s^{-1}V)_n=V_{n+1}$. Write $s^{-1}v \in (s^{-1}V)_n$ for $v \in V_{n+1}$ and the map $s^{-1}: V \rightarrow s^{-1}V$, $v\mapsto s^{-1}v$ is a graded map of degree $-1$. In the context of cohomologically graded spaces, as used above, the suspension $s$ and desuspension $s^{-1}$ have degree $-1$ and $1$ respectively.

\medskip

\subsection{Coefficients from regular Nijenhuis bimodules to general Nijenhuis bimodules}\ \label{Sub: cohomology proofs}

Let $M=(M, P_M)$ be a Nijenhuis bimodule over a Nijenhuis associative algebra $A=(A, m, P)$.
According to Proposition~\ref{Prop: trivial extension of Nijenhuis bimodule}, the semi-direct product $A \ltimes M$ of Nijenhuis associative algebra $A$ by its Nijenhuis bimodule $M$ is also a Nijenhuis associative algebra. Hence, we get the cochain complex $(\C^{\bullet}_{\NjA}(A \ltimes M), \delta^{\bullet}_{\NjA})$ of the Nijenhuis associative algebra $A \ltimes M$, where
\begin{align*}
	\C^0_{\NjA}(A \ltimes M) &= \C^0_\Alg(A \ltimes M) = A \oplus M,\\
	\C^n_{\NjA}(A \ltimes M) &= \C^n_\Alg(A \ltimes M) \oplus \C^{n-1}_{\NjO}(A \ltimes M) \\
	&= \Hom((A \oplus M)^{n}, A \oplus M) \oplus \Hom((A \oplus M)^{n-1}, A \oplus M), \text{ for } n\geqslant 1.
\end{align*}

In order to obtain the cochain complex of Nijenhuis associative algebra $A$ with coefficients in Nijenhuis bimodules $M$, one also needs to consider the obvious embedding
$$\iota: \C^{\bullet}_{\NjA}(A, M) \hookrightarrow \C^{\bullet}_{\NjA}(A \ltimes M), \ f \mapsto \overline{f}.$$

\begin{prop}
	With above notations, the image $\mathrm{Im}(\iota)$ is a subcomplex of $(\C^{\bullet}_{\NjA}(A \ltimes M), \delta^{\bullet}_{\NjA})$.
\end{prop}
\begin{proof}
	Let $(f,g) \in \C^n_{\NjA}(A, M)$ for $n \geqslant 0$. Direct calculation can obtain:
	\begin{align*}
		\delta_{\NjA}^n(\iota(f,g))
		=&\ \delta_{\NjA}^n(\overline{f}, \overline{g})\\
		=&\ (\delta_{\Alg}^n(\overline{f}), -\Phi^n(\overline{f})-\delta_{\NjO}^{n-1}(\overline{g})) \\
		=&\ \iota(\delta_{\Alg, M}^n(f), -\Phi_{M}^n(f)-\delta_{\NjO, M}^{n-1}(g))\\
		=&\ \iota(\delta_{\NjA, M}^n(f,g)).
	\end{align*}
\end{proof}

Let $p: \mathrm{Im}(\iota) \twoheadrightarrow \C^{\bullet}_{\NjA}(A, M), \overline{f} \mapsto f$, be the obvious projection. Using the fact $\iota \circ p = \id_{\mathrm{Im}(\iota)}$, we find that $\delta^{\bullet}_{\NjA, M} = p \circ \delta^{\bullet}_{\NjA} \circ \iota$ is indeed a differential on the space $\C^{\bullet}_{\NjA}(A, M)$, which justify the Definition~\ref{Def: definition of Nijenhuis cohomology}. Moreover, Lemma~\ref{Lem: cochain complex of Nij operator} and Proposition~\ref{Prop: Chain map Phi} are direct corollaries as well.

\medskip

\section{The Koszul dual homotopy cooperad of the operad for Nijenhuis associative algebras} \label{Section: Koszul dual cooperad NjA}

In this section, we construct a homotopy cooperad $\NjAoperad^\ac$, which serves as the Koszul dual of $\NjAoperad$, the operad for Nijenhuis associative algebras. This is justified by the fact that the cobar construction of $\NjAoperad^\ac$ is exactly $\NjAoperad_\infty$, the minimal model of $\NjAoperad$, as will be proved in Section~\ref{Section: minimal model NjA}.

Define a graded collection $\mathscr{S}(\NjAoperad^\ac)$ by
$$\mathscr{S}(\NjAoperad^\ac)(n)=\bfk u_n\oplus \bfk v_n$$
with $|u_n|=0, |v_n|=1$ for $n\geqslant 1$. Now, we put a coaugmented homotopy cooperad structure on $\mathscr{S}(\NjAoperad^\ac)$. Firstly, consider trees of arity $n\geqslant 1$ in the following list:
\begin{enumerate}

\item Trees of type (I) which is of weight $2$: for each $1\leqslant j\leqslant n$ and $1\leqslant i\leqslant n-j+1$, there exists such a tree which can be visualized as
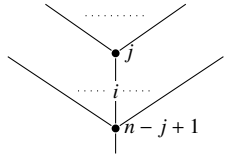
\begin{eqnarray*}
\begin{tikzpicture}[scale=1, descr/.style={fill=white}]
\tikzstyle{every node}=[thick, minimum size=3pt, inner sep=1pt]
\node(v0) at (0, 0)[fill=black, circle, label=right:{\tiny $ n-j+1$}]{};
\node(v1-1) at (-1.5, 1){};
\node(v1-2) at(0, 1)[fill=black, circle, label=right:{\tiny $\tiny j$}]{};
\node(v1-3) at(1.5, 1){};
\node(v2-1)at (-1, 1.7){};
\node(v2-2) at(1, 1.7){};
\draw(v0)--(v1-1);
\draw(v0)--(v1-3);
\draw(v1-2)--(v2-1);
\draw(v1-2)--(v2-2);
\draw[dotted](-0.4, 1.5)--(0.4, 1.5);
\draw[dotted](-0.5, 0.5)--(-0.1, 0.5);
\draw[dotted](0.1, 0.5)--(0.5, 0.5);
\path[-, font=\scriptsize]
(v0) edge node[descr]{{\tiny$i$}} (v1-2);

\node(vroot) at (0, -0.4){};
\draw(v0)--(vroot);
\end{tikzpicture}
\end{eqnarray*}

\item Trees of type (II) which is of weight $p+1\geqslant 3$ and height $ t+2 $ with $0 \leqslant t \leqslant p \leqslant n$, and there exists a unique vertex in the first $ t+1 $ levels: in these trees, there are numbers $1\leqslant k_t<\cdots<k_{p-1}\leqslant p$, $r_1, \dots, r_p\geqslant 1$ and $1\leqslant i_{h} \leqslant r_{h} $ for $1\leqslant h \leqslant t $, such that
the $ q $-th vertex has arity $ r_{q} $ for all $1 \leqslant q \leqslant t $, the $ m $-th vertex is connected to the $ i_{m-1} $-th leaf of the $ (m-1) $-th vertex for all $ 2 \leqslant m \leqslant t $, the $ (t+1) $-th vertex has arity $ p $ and is connected to the $ i_{t} $-th leaf of the $ t $-th vertex, and the other vertices connected to the $k_t$-th (resp. $k_{t+1}$-th, $\dots, k_{p-1}$-th) leaf of the $ (t+1) $-th vertex and with arity $r_{t+1}$ (resp. $r_{t+2}, \dots, r_p$), so $r_1+\cdots+r_p=n \geqslant 2$. These trees can be visualized as:

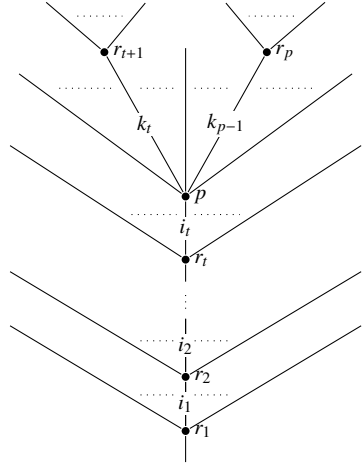
\begin{eqnarray*}
	\begin{tikzpicture}[scale=1.2, descr/.style={fill=white}]
	\tikzstyle{every node}=[thick, minimum size=3pt, inner sep=1pt]
	\begin{scope}[shift={(0,0.8)}, scale=1]
		\node(v-1) at (0, -2.2)[circle, fill=black, label=right:{\tiny $r_1$}]{};
		\node(v0-1) at (-2, -1){ };
		\node(v0-3) at(2, -1){ };	

		\node(vroot) at (0, -2.6){};
		\draw(v-1)--(vroot);
					
		\node(v-1-2) at (0, -1.6)[circle, fill=black, label=right:{\tiny $r_2$}]{};
		\node(v0-1-2) at (-2, -0.4){ };
		\node(v0-3-2) at(2, -0.4){ };
		\node(v0-2-1-2) at(0, -0.9)[]{};
		\node(v0-2-2) at(0, -0.6)[]{};
		\draw(v-1-2)--(v0-1-2);
		\draw(v-1-2)--(v0-2-1-2);
		\draw(v-1-2)--(v0-3-2);
		\draw(v-1)--(v-1-2);
		\draw[dotted](v0-2-1-2)--(v0-2-2);
		\path[-, font=\scriptsize]

		(v-1) edge node[descr]{{\tiny$i_{1}$}} (v-1-2);
		\path[-, font=\scriptsize]
		(v-1-2) edge node[descr]{{\tiny$i_{2}$}} (v0-2-1-2);

		\draw[dotted](-0.5, -1.8)--(-0.1, -1.8);
		\draw[dotted](0.1, -1.8)--(0.5, -1.8);

		\draw[dotted](-0.5, -1.2)--(-0.1, -1.2);
		\draw[dotted](0.1, -1.2)--(0.5, -1.2);
	\end{scope}
		\node(v0) at (0, 0.5)[circle, fill=black, label=right:{\tiny $r_t$}]{};
		\node(v1-1) at (-2, 1.8){ };
		\node(v1-2) at(0, 1.2)[circle, fill=black, label=right:{\tiny $p$}]{};
		\draw[dotted](-0.6, 1)--(-0.1, 1);
		\draw[dotted](0.1, 1)--(0.6, 1);

		\node(v1-3) at(2, 1.8){ };
		\node(v2-1) at(-1.9, 2.6){ };
		\node(v2-2) at (-0.9, 2.8)[circle, fill=black, label=right:{\tiny $r_{t+1}$}]{};
		\node(v2-3) at (0, 2.9){};
		\node(v2-4) at(0.9, 2.8)[circle, fill=black, label=right:{\tiny $r_{p}$}]{};
		\node(v2-5) at(1.9, 2.6){ };
		\node(v3-1) at (-1.6, 3.4){ };
		\node(v3-2) at (-0.4, 3.4){ };
		\node(v3-3) at (0.4, 3.4){ };
		\node(v3-4) at(1.6, 3.4){ };

		\draw(v0-2-2)--(v0);
		\draw(v0)--(v1-1);
		\draw(v0)--(v1-3);
		\draw(v-1)--(v0-1);
		\draw(v-1)--(v0-3);
		
		\path[-, font=\scriptsize]
		(v0) edge node[descr]{{\tiny$i_{t}$}} (v1-2);
		
		\draw(v1-2)--(v2-1);
		\draw(v1-2)--(v2-3);
		\draw(v1-2)--(v2-5);
		
		\path[-, font=\scriptsize]
		(v1-2) 	edge node[descr]{{\tiny$k_t$}} (v2-2)
		edge node[descr]{{\tiny$k_{p-1}$}} (v2-4);
		\draw(v2-2)--(v3-1);
		\draw(v2-2)--(v3-2);
		\draw(v2-4)--(v3-3);
		\draw(v2-4)--(v3-4);

		\draw[dotted](-0.5, 2.4)--(-0.1, 2.4);
		\draw[dotted](0.1, 2.4)--(0.5, 2.4);
		\draw[dotted](-1.4, 2.4)--(-0.8, 2.4);
		\draw[dotted](1.4, 2.4)--(0.8, 2.4);
		\draw[dotted](-1.2, 3.2)--(-0.7, 3.2);
		\draw[dotted](1.2, 3.2)--(0.7, 3.2);
	\end{tikzpicture}
\end{eqnarray*}
\end{enumerate}

Now, we define a family of operations $\{\Delta_T: \mathscr{S}(\NjAoperad^\ac)\rightarrow \mathscr{S}({\NjAoperad^\ac})^{\ot T}\}_{T\in \frakT}$ as follows:
\begin{enumerate}
	\item For a tree $T$ of type $\mathrm{(I)}$ and $n \geqslant 1$,
	 define $\Delta_T(u_n)=u_{n-j+1}\ot u_j$, which can be drawn as
	
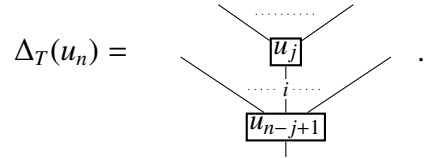
\begin{eqnarray*}
	\begin{tikzpicture}[scale=1, descr/.style={fill=white}]
	\tikzstyle{every node}=[thick, minimum size=5pt, inner sep=1pt]
	\node(v-2) at(-2, 1)[minimum size=0pt, label=left:{$\Delta_T(u_n)=$}]{};
	\node(sigma-1) at (1.8, 0.9){. };
	\node(v0) at (0, 0)[draw, rectangle]{{\small $u_{n-j+1}$}};
	\node(v1-1) at (-1.5, 1){};
	\node(v1-2) at(0, 1)[draw, rectangle]{\small$u_j$};
	\node(v1-3) at(1.5, 1){};
	\node(v2-1)at (-1, 1.7){};
	\node(v2-2) at(1, 1.7){};
	\draw(v0)--(v1-1);
	\draw(v0)--(v1-3);
	\draw(v1-2)--(v2-1);
	\draw(v1-2)--(v2-2);
	\draw[dotted](-0.4, 1.5)--(0.4, 1.5);
	\draw[dotted](-0.5, 0.5)--(-0.1, 0.5);
	\draw[dotted](0.1, 0.5)--(0.5, 0.5);
	\path[-, font=\scriptsize]
	(v0) edge node[descr]{{\tiny$i$}} (v1-2);

	\node(vroot) at (0, -0.5){};
	\draw(v0)--(vroot);
	\end{tikzpicture}
	\end{eqnarray*}
	Define
	$$\Delta_T(v_n)=
	\left\{\begin{array}{ll} v_n\ot u_1, & j=1, \\
	0, & 2\leqslant j\leqslant n-1, \\
	u_1\ot v_n, & j=n,
	\end{array}\right.$$
	which can be pictured as
	\begin{enumerate}
		\item when $j=1$,
		$$
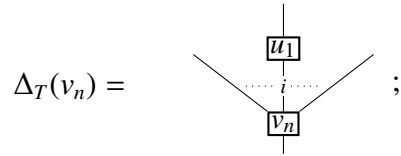
\begin{tikzpicture}[scale=1, descr/.style={fill=white}]
			\tikzstyle{every node}=[thick, minimum size=5pt, inner sep=1pt]
			\node(v-1) at(-2, 0.5)[minimum size=0pt, label=left:{ $\Delta_T(v_n)=$}]{};
			\node(dd) at(1.5, 0.5){;};
			\node(v0) at (0, 0)[draw, rectangle]{\small$v_n$};
			\node(v1-1) at (-1.3, 1){};
			\node(v1-2) at(0, 1)[draw, rectangle]{\small$u_1$};
			\node(v1-3) at(1.3, 1){};
			\node(v2-1)at (0, 1.7){};
			\draw(v0)--(v1-1);
			\draw(v0)--(v1-3);
			\draw(v1-2)--(v2-1);
			\draw[dotted](-0.5, 0.5)--(-0.1, 0.5);
			\draw[dotted](0.1, 0.5)--(0.5, 0.5);
			\path[-, font=\scriptsize]
			(v0) edge node[descr]{{\tiny$i$}} (v1-2);
			\node(vroot) at (0, -0.5){};
			\draw(v0)--(vroot);
		\end{tikzpicture}$$
		\item when $2\leqslant j\leqslant n-1$, $\Delta_T(v_n)=0$;
		\item when $j=n$,
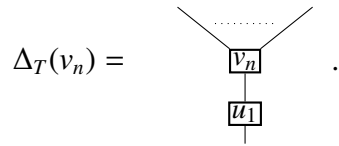
		$$\begin{tikzpicture}[scale=1, descr/.style={fill=white}]
			\tikzstyle{every node}=[thick, minimum size=5pt, inner sep=1pt]
			\node(va) at(2, 0.7)[minimum size=0pt, label=left:{ $\Delta_T(v_n)=$}]{};
			\node(dd) at(4.7, 0.6){.};
			\node(ve) at (3.5, 0)[draw, rectangle]{\small $u_1$};
			\begin{scope}[shift={(0,-0.3)}, scale=1]
				\node(ve1) at (3.5, 1)[draw, rectangle]{\small $v_n$};
				\node(ve2-1) at(2.5, 1.8){};
				\node(ve2-2) at(4.5, 1.8){};
				\draw[dotted](3.1, 1.5)--(3.9, 1.5);
			\end{scope}
			\draw(ve)--(ve1);
			\draw(ve1)--(ve2-1);
			\draw(ve1)--(ve2-2);

			\node(vroot) at (3.5, -0.5){};
			\draw(ve)--(vroot);
		\end{tikzpicture}$$
	\end{enumerate}

	\item For a tree $T$ of type $\mathrm{(II)}$ and $n \geqslant 2$, define
	
	
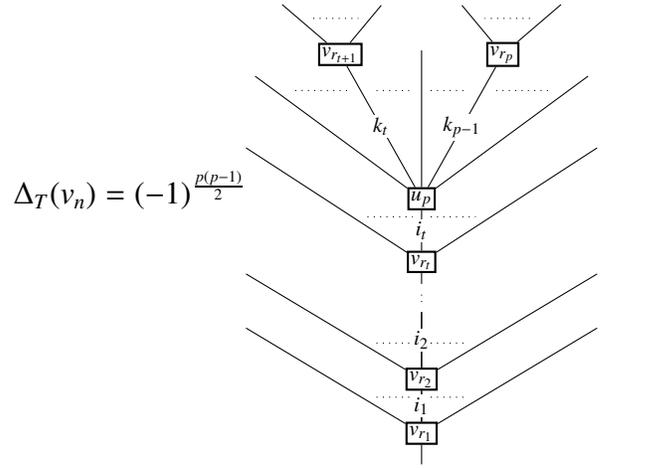
\begin{eqnarray*}
		\begin{tikzpicture}[scale=1.2, descr/.style={fill=white}]
			\tikzstyle{every node}=[thick, minimum size=3pt, inner sep=1pt]
	
				\node(v-0-1) at(-4.6, 1.3)[minimum size=0pt, label=right:{$\Delta_T(v_n)=(-1)^\frac{p(p-1)}{2}$}]{};
				\node(v-ddd) at(2.5, 1.1){.};
	
				\begin{scope}[shift={(0,0.8)}, scale=1]
					\node(v-1) at (0, -2.2)[rectangle, draw]{\tiny $v_{r_{1}}$};
					\node(v0-1) at (-2, -1){ };
					\node(v0-3) at(2, -1){ };	
								
					\node(v-1-2) at (0, -1.6)[rectangle, draw]{\tiny $v_{r_{2}}$};
					\node(v0-1-2) at (-2, -0.4){ };
					\node(v0-3-2) at(2, -0.4){ };
					\node(v0-2-1-2) at(0, -0.8)[]{};
					\node(v0-2-2) at(0, -0.6)[]{};
					\draw(v-1-2)--(v0-1-2);
					\draw(v-1-2)--(v0-2-1-2);
					\draw(v-1-2)--(v0-3-2);
					\draw(v-1)--(v-1-2);
					\draw[dotted](v0-2-1-2)--(v0-2-2);
					\path[-, font=\scriptsize]
			
					(v-1) edge node[descr]{{\tiny$i_{1}$}} (v-1-2);
					\path[-, font=\scriptsize]
					(v-1-2) edge node[descr]{{\tiny$i_{2}$}} (v0-2-1-2);
			
					\draw[dotted](-0.5, -1.8)--(-0.1, -1.8);
					\draw[dotted](0.1, -1.8)--(0.5, -1.8);
			
					\draw[dotted](-0.5, -1.2)--(-0.1, -1.2);
					\draw[dotted](0.1, -1.2)--(0.5, -1.2);

					\node(vroot) at (0, -2.6){};
					\draw(v-1)--(vroot);
				\end{scope}
					\node(v0) at (0, 0.5)[rectangle, draw]{\tiny $v_{r_{t}}$};
					\node(v1-1) at (-2, 1.8){ };
					\node(v1-2) at(0, 1.2)[rectangle, draw]{\tiny $u_{p}$};
					\draw[dotted](-0.6, 1)--(-0.1, 1);
					\draw[dotted](0.1, 1)--(0.6, 1);
			
					\node(v1-3) at(2, 1.8){ };
					\node(v2-1) at(-1.9, 2.6){ };
					\node(v2-2) at (-0.9, 2.8)[rectangle, draw]{\tiny $v_{r_{t+1}}$};
					\node(v2-3) at (0, 2.9){};
					\node(v2-4) at(0.9, 2.8)[rectangle, draw]{\tiny $v_{r_{p}}$};
					\node(v2-5) at(1.9, 2.6){ };
					\node(v3-1) at (-1.6, 3.4){ };
					\node(v3-2) at (-0.4, 3.4){ };
					\node(v3-3) at (0.4, 3.4){ };
					\node(v3-4) at(1.6, 3.4){ };
			
					\draw(v0-2-2)--(v0);
					\draw(v0)--(v1-1);
					\draw(v0)--(v1-3);
					\draw(v-1)--(v0-1);
					\draw(v-1)--(v0-3);
					
					\path[-, font=\scriptsize]
					(v0) edge node[descr]{{\tiny$i_{t}$}} (v1-2);
					
					\draw(v1-2)--(v2-1);
					\draw(v1-2)--(v2-3);
					\draw(v1-2)--(v2-5);
					
					\path[-, font=\scriptsize]
					(v1-2) 	edge node[descr]{{\tiny$k_t$}} (v2-2)
					edge node[descr]{{\tiny$k_{p-1}$}} (v2-4);
					\draw(v2-2)--(v3-1);
					\draw(v2-2)--(v3-2);
					\draw(v2-4)--(v3-3);
					\draw(v2-4)--(v3-4);
			
					\draw[dotted](-0.5, 2.4)--(-0.1, 2.4);
					\draw[dotted](0.1, 2.4)--(0.5, 2.4);
					\draw[dotted](-1.4, 2.4)--(-0.8, 2.4);
					\draw[dotted](1.4, 2.4)--(0.8, 2.4);
					\draw[dotted](-1.2, 3.2)--(-0.7, 3.2);
					\draw[dotted](1.2, 3.2)--(0.7, 3.2);
			\end{tikzpicture}
	\end{eqnarray*}

	\item For any other tree $T $ in $\frakT$, define $\Delta_T = 0$.
\end{enumerate}

In the proof of the following result, we need the pre-Jacobi identity \cite{Ger63, GV95, Get93,CGWZ24}:
for any homogeneous elements $f, g_1, \dots, g_m, h_1, \dots, h_n$ in a graded operad $\calp$, we have
	\begin{align} \label{Eq. pre-Jacobi identity}
			& \Big(f \{g_1, \dots, g_m\}\Big)\{h_1, \dots, h_n\}=\\
			\notag
			& \sum\limits_{0\leqslant i_1\leqslant j_1\leqslant i_2\leqslant j_2\leqslant \dots \leqslant i_m\leqslant j_m\leqslant n}(-1)^{\theta}
			f \Big\{h_{1}, \dots, h_{i_1}, g_1\{h_{i_1+1}, \dots, h_{j_1}\}, \dots, g_m\{h_{i_m+1}, \dots, h_{j_m} \}, h_{j_m+1}, \dots, h_{n} \Big\},
	\end{align}
	where $\theta = \sum\limits_{k=1}^m (|g_k|)(\sum\limits_{j=1}^{i_k}|h_j|)$ and the operation $ -\{-\} $ is the brace operation introduced in \cite{GV95} (see also \cite[Definition 3.6]{CGWZ24}).

\begin{prop} \label{prop: homotopy cooperad cobar construction NjA}
	The graded collection $\mathscr{S}(\NjAoperad^\ac)$ endowed with the operations $\{\Delta_T\}_{T\in \frakT}$ introduced above forms a coaugmented homotopy cooperad, whose strict counit is the natural projection $\varepsilon:\mathscr{S}(\NjAoperad^\ac)\twoheadrightarrow \bfk u_1\cong \cali$ and the coaugmentation is just the natural embedding $\eta:\cali\cong \bfk u_1\hookrightarrow \mathscr{S}(\NjAoperad^\ac)$. 	
\end{prop}

\begin{proof}
	Equivalently, one needs to show that the induced derivation $\partial$ on the cobar construction of $\mathscr{S}(\NjAoperad^\ac)$, i.e., the free operad generated by $s^{-1}\overline{\mathscr{S}(\NjAoperad^\ac)}$, is a differential, that is, $\partial^2=0$.

	Denote $s^{-1}u_n, n\geqslant 2$ (resp. $s^{-1}v_n, n\geqslant 1$) by $x_n$ (resp. $y_n$) which are the generators of $s^{-1}\overline{\mathscr{S}(\NjAoperad^\ac)}$. Notice that $|x_n|=-1$ and $|y_n|=0$. By the definition of cobar construction of coaugmented homotopy cooperads, the action of differential $\partial$ on generators $x_n, y_n$ is given by the following formulas:
	\begin{eqnarray} \label{Eq: partial xn NjA}
		\partial(x_n) = -\sum\limits_{j=2}^{n-1}x_{n-j+1}\{x_j\}, \ n\geqslant 2,
	\end{eqnarray}
	\begin{eqnarray} \label{Eq: partial yn NjA} \partial(y_n) =
		-\sum_{\substack{ r_1+\cdots+r_p =n\\ r_1, \dots, r_p \geqslant 1 \\	2 \leqslant p \leqslant n}}
		\sum_{t=0}^{p}
		(-1)^{t}
		y_{r_{1}} \bigg\{ \cdots \Big\{ y_{r_{t}} \big \{ x_{p} \{ y_{r_{t+1}}, \dots, y_{r_{p}} \} \big\} \Big\} \cdots \bigg\}, \ n \geqslant 1.
	\end{eqnarray}
	Note that $\partial(x_{2})=0$, $\partial(y_1)=0$.

	We just need to prove that $\partial^2=0$ holds on generators $x_n, n\geqslant 2$ and $y_n, n\geqslant 1$, which can be checked by direct computations.
	As $\partial$ is a derivation with respect to $-\{-\}$, we have
	\begin{eqnarray*}
		\partial^2(x_n)&\stackrel{\eqref{Eq: partial xn NjA}}{=}&\partial\Big(-\sum\limits_{j=2}^{n-1}x_{n-j+1}\{x_j\}\Big)\\
		&=&-\sum\limits_{j=2}^{n-1}\partial(x_{n-j+1})\{x_j\}+\sum_{j=2}^{n-1}x_{n-j+1}\{\partial(x_j)\}\\
		&\stackrel{\eqref{Eq: partial xn NjA}}{=}&\sum_{i+j+k-2=n \atop
			2\leqslant i, j, k\leqslant n-2}(x_i\{x_j\})\{x_k\}-\sum_{i+j+k-2=n \atop
			2\leqslant i, j, k\leqslant n-2}x_i\{x_j\{x_k\}\}\\
		&\stackrel{\eqref{Eq. pre-Jacobi identity}}{=}&\sum_{i+j+k-2=n \atop
			2\leqslant i, j, k\leqslant n-2}x_i\{x_j\{x_k\}\}+\sum_{i+j+k-2=n \atop
			2\leqslant i, j, k\leqslant n-2}x_i\{x_j, x_k\}-\sum_{i+j+k-2=n \atop
			2\leqslant i, j, k\leqslant n-2}x_i\{x_k, x_j\}-\sum_{i+j+k-2=n \atop
			2\leqslant i, j, k\leqslant n-2}x_i\{x_j\{x_k\}\}\\
		&=&0,
	\end{eqnarray*}
	and
	\begin{eqnarray*}
		\partial^{2}(y_{n})
		&\stackrel{\eqref{Eq: partial yn NjA}}{=}
		& \partial \Big(
		-\sum_{\substack{ r_1+\cdots+r_p =n\\ r_1, \dots, r_p \geqslant 1 \\	2 \leqslant p \leqslant n}}
		\sum_{t=0}^{p}
		(-1)^{t}
		y_{r_{1}} \{ \cdots \{ y_{r_{t}} \{ x_{p} \{ y_{r_{t+1}}, \dots, y_{r_{p}} \} \} \} \cdots \}
		\Big)\\
		&=& -\sum_{\substack{ r_1+\cdots+r_p =n\\ r_1, \dots, r_p \geqslant 1 \\	2 \leqslant p \leqslant n-1}}
		\sum_{t=1}^{p}
		\sum_{i=1}^{t}
		(-1)^{t}
		y_{r_{1}}
		\{\cdots
		\{
		y_{r_{i-1}}\{
		\partial(y_{r_{i}})
		\{
		y_{r_{i+1}}
		\{
		\cdots \{ y_{r_{t}} \{ x_{p} \{ y_{r_{t+1}}, \dots, y_{r_{p}} \} \} \} \cdots
		\}
		\}
		\}
		\}
		\cdots
			\}\\
		&& -\sum_{\substack{ r_1+\cdots+r_p =n\\ r_1, \dots, r_p \geqslant 1 \\	3 \leqslant p \leqslant n}}
		\sum_{t=0}^{p}
		(-1)^{t}
		y_{r_{1}} \{ \cdots \{ y_{r_{t}} \{
		\partial(x_{p})
			\{ y_{r_{t+1}}, \dots, y_{r_{p}} \} \} \} \cdots \}\\
		&& +\sum_{\substack{ r_1+\cdots+r_p =n\\ r_1, \dots, r_p \geqslant 1 \\	2 \leqslant p \leqslant n-1}}
		\sum_{t=0}^{p-1}
		\sum_{i=t+1}^{p}
		(-1)^{t}
		y_{r_{1}} \{ \cdots \{ y_{r_{t}} \{ x_{p}
		\{ y_{r_{t+1}}, \dots, y_{r_{i-1}}, \partial(y_{r_{i}}), y_{r_{i+1}}, \dots, y_{r_{p}} \}
			\} \} \cdots \}\\
		&\stackrel{\eqref{Eq: partial xn NjA} \eqref{Eq: partial yn NjA}}{=}& \sum_{\substack{ r_1+\cdots+r_p =n\\ r_1, \dots, r_p \geqslant 1 \\	2 \leqslant p \leqslant n-1}}
		\sum_{t=1}^{p}
		\sum_{i=1}^{t}%
		\sum_{\substack{ l_1+\cdots+l_q =r_{i}\\ l_1, \dots, l_q \geqslant 1 \\	2 \leqslant q \leqslant r_{i} \leqslant n-1}}
		\sum_{s=0}^{q}
		(-1)^{s+t}\\
		&&\begin{aligned}
			y_{r_{1}}
			\{\cdots
			\{
			y_{r_{i-1}}\{ &\\
			\Big(
			y_{l_{1}} \{\cdots \{ y_{l_{s}} \{ & x_{q} \{ y_{l_{s+1}}, \dots, y_{l_{q}} \} \} \} \cdots \}
			\Big)
			\{ %
			y_{r_{i+1}}
			\{
			\cdots \{ y_{r_{t}} \{ x_{p} \{ y_{r_{t+1}}, \dots, y_{r_{p}} \} \} \} \cdots
			\}\}\\ %
			& \}\}\cdots\}
			\end{aligned}\\
		&& +\sum_{\substack{ r_1+\cdots+r_p =n\\ r_1, \dots, r_p \geqslant 1 \\	3 \leqslant p \leqslant n}}
		\sum_{t=0}^{p}
		\sum_{j=2}^{p-1}
		(-1)^{t}
		y_{r_{1}} \{ \cdots \{ y_{r_{t}} \{ 
		\Big(
		x_{p-j+1}\{x_{j}\}
		\Big) 
		\{ y_{r_{t+1}}, \dots, y_{r_{p}} \} \} \} \cdots \}\\
		&& -\sum_{\substack{ r_1+\cdots+r_p =n\\ r_1, \dots, r_p \geqslant 1 \\	2 \leqslant p \leqslant n-1}}
		\sum_{t=0}^{p-1}
		\sum_{i=t+1}^{p}
		\sum_{\substack{ l_1+\cdots+l_q =r_{i}\\ l_1, \dots, l_q \geqslant 1 \\	2 \leqslant q \leqslant r_{i} \leqslant n-1}}
		\sum_{s=0}^{q}
		(-1)^{s+t}\\
		&& y_{r_{1}} \{ \cdots \{ y_{r_{t}} \{ x_{p}
		\{ y_{r_{t+1}}, \dots, y_{r_{i-1}}, 
		y_{l_{1}} \{\cdots \{ y_{l_{s}} \{ x_{q} \{ y_{l_{s+1}}, \dots, y_{l_{q}} \} \} \} \cdots \}
		, y_{r_{i+1}}, \dots, y_{r_{p}} \}
		\} \} \cdots \}\\
		&\stackrel{\eqref{Eq. pre-Jacobi identity}}{=}& \sum_{\substack{ r_1+\cdots+r_p =n\\ r_1, \dots, r_p \geqslant 1 \\	2 \leqslant p \leqslant n-2}}
		\sum_{t=1}^{p}
		\sum_{i=1}^{t}%
		\sum_{\substack{ l_1+\cdots+l_q =r_{i}\\ l_1, \dots, l_q \geqslant 1 \\	2 \leqslant q \leqslant r_{i}-1 \leqslant n-2}}
		\sum_{s=1}^{q}
		\sum_{j=1}^{s}
		(-1)^{s+t}\\
		&& \underbrace{\begin{aligned}
		y_{r_{1}}
		\{\cdots
		\{
		y_{r_{i-1}}\{
		y_{l_{1}}
		\{\cdots \{
		y_{l_{j}}
		\{&\\
			y_{l_{j+1}} \{
		\cdots
		\{
		y_{l_{s}}
		\{
		x_{q}
		\{
		y_{l_{s+1}}, \dots, y_{l_{q}}
		\}
		\}
		\}
		\cdots
		\}
		&, 
		y_{r_{i+1}}
		\{
		\cdots \{ y_{r_{t}} \{ x_{p} \{ y_{r_{t+1}}, \dots, y_{r_{p}} \} \} \} \cdots
		\} \\
		& \}\}\cdots\} 	\}\}\cdots\}\\
		\end{aligned}}_{(\text{I})}\\
		&&- \sum_{\substack{ r_1+\cdots+r_p =n\\ r_1, \dots, r_p \geqslant 1 \\	2 \leqslant p \leqslant n-2}}
		\sum_{t=1}^{p}
		\sum_{i=1}^{t}%
		\sum_{\substack{ l_1+\cdots+l_q =r_{i}\\ l_1, \dots, l_q \geqslant 1 \\	2 \leqslant q \leqslant r_{i}-1 \leqslant n-2}}
		\sum_{s=1}^{q}
		\sum_{j=1}^{s}
		(-1)^{s+t}\\
		&& \underbrace{\begin{aligned}
		y_{r_{1}}
		\{\cdots
		\{
		y_{r_{i-1}}\{
		y_{l_{1}}
		\{\cdots \{
		y_{l_{j}}
		\{&\\
		y_{r_{i+1}}
		\{
		\cdots \{ y_{r_{t}} \{ x_{p} \{ y_{r_{t+1}}, \dots, y_{r_{p}} \} \} \} \cdots
		\} 
		&, 
		y_{l_{j+1}} \{
		\cdots
		\{
		y_{l_{s}}
		\{
		x_{q}
		\{
		y_{l_{s+1}}, \dots, y_{l_{q}}
		\}
		\}
		\}
		\cdots
		\}\\
		& \}\}\cdots\} 	\}\}\cdots\}
		\end{aligned}}_{(\text{II})}\\
		&& + \sum_{\substack{ r_1+\cdots+r_p =n\\ r_1, \dots, r_p \geqslant 1 \\	2 \leqslant p \leqslant n-1}}
		\sum_{t=1}^{p}
		\sum_{i=1}^{t}%
		\sum_{\substack{ l_1+\cdots+l_q =r_{i}\\ l_1, \dots, l_q \geqslant 1 \\	2 \leqslant q \leqslant r_{i} \leqslant n-1}}
		\sum_{s=0}^{q-1}
		\sum_{m=1}^{q-s}
		(-1)^{s+t}\\
		&& \underbrace{\begin{aligned}
		y_{r_{1}}
		\{\cdots
		\{
		y_{r_{i-1}}\{
		y_{l_{1}}
		\{\cdots \{
		y_{l_{s}}
		\{&\\
		x_{q}
		\{
		y_{l_{s+1}}, \dots, y_{l_{s+m-1}}, & y_{l_{s+m}}
		\{
		y_{r_{i+1}}
		\{
		\cdots \{ y_{r_{t}} \{ x_{p} \{ y_{r_{t+1}}, \dots, y_{r_{p}} \} \} \} \cdots
		\}
		\}
		, y_{l_{s+m+1}}, \dots, y_{l_{q}}
		\}
		\\
		& \}\}\cdots\} 	\}\}\cdots\}
		\end{aligned}}_{(\text{III})}\\
		&& + \sum_{\substack{ r_1+\cdots+r_p =n\\ r_1, \dots, r_p \geqslant 1 \\	2 \leqslant p \leqslant n-1}}
		\sum_{t=1}^{p}
		\sum_{i=1}^{t}%
		\sum_{\substack{ l_1+\cdots+l_q =r_{i}\\ l_1, \dots, l_q \geqslant 1 \\	2 \leqslant q \leqslant r_{i} \leqslant n-1}}
		\sum_{s=1}^{q}
		\sum_{m=0}^{q-s}
		(-1)^{s+t}\\
		&& \underbrace{\begin{aligned}
		y_{r_{1}}
		\{\cdots
		\{
		y_{r_{i-1}}\{
		y_{l_{1}}
		\{\cdots \{
		y_{l_{s}}
		\{&\\
		x_{q}
		\{
		y_{l_{s+1}}, \dots, & y_{l_{s+m}}
		,
		y_{r_{i+1}}
		\{
		\cdots \{ y_{r_{t}} \{ x_{p} \{ y_{r_{t+1}}, \dots, y_{r_{p}} \} \} \} \cdots
		\}
		, y_{l_{s+m+1}}, \dots, y_{l_{q}}
		\}
		\\
		& \}\}\cdots\} 	\}\}\cdots\}
		\end{aligned}}_{(\text{IV})}\\
		&& + \sum_{\substack{ r_1+\cdots+r_p =n\\ r_1, \dots, r_p \geqslant 1 \\	3 \leqslant p \leqslant n}}
		\sum_{t=0}^{p-1}
		\sum_{j=2}^{p-1}
		\sum_{i=0}^{p-t-1}
		\sum_{\substack{q=\max\{1, j-t\}}}^{j}
		(-1)^{t}\\
		&& \underbrace{y_{r_{1}} \{ \cdots \{ y_{r_{t}} \{ 
		x_{p-j+1}\{
		y_{r_{t+1}}, \dots, y_{r_{t+i}}, x_{j} \{ y_{r_{t+i+1}}, \dots, y_{r_{t+i+q}} \},
		y_{r_{t+i+q+1}}, \dots, y_{r_{p}} 
			\} \} \} \cdots \}}_{(\text{V})}\\
			&& + \sum_{\substack{ r_1+\cdots+r_p =n\\ r_1, \dots, r_p \geqslant 1 \\	3 \leqslant p \leqslant n}}
			\sum_{t=2}^{p}
			\sum_{j=2}^{\min\{t,p-1\}}
			\sum_{i=0}^{p-t}
			(-1)^{t}\\
			&& \underbrace{y_{r_{1}} \{ \cdots \{ y_{r_{t}} \{ 
			x_{p-j+1}\{
			y_{r_{t+1}}, \dots, y_{r_{t+i}}, x_{j}, y_{r_{t+i+1}}, \dots, y_{r_{p}} 
			\} \} \} \cdots \}}_{(\text{VI})}\\
			&& -\sum_{\substack{ r_1+\cdots+r_p =n\\ r_1, \dots, r_p \geqslant 1 \\	2 \leqslant p \leqslant n-1}}
			\sum_{t=0}^{p-1}
			\sum_{i=t+1}^{p}
			\sum_{\substack{ l_1+\cdots+l_q =r_{i}\\ l_1, \dots, l_q \geqslant 1 \\	2 \leqslant q \leqslant r_{i} \leqslant n-1}}
			\sum_{s=0}^{q}
			(-1)^{s+t}\\
			&& \underbrace{y_{r_{1}} \{ \cdots \{ y_{r_{t}} \{ x_{p}
			\{ y_{r_{t+1}}, \dots, y_{r_{i-1}}, 
			y_{l_{1}} \{\cdots \{ y_{l_{s}} \{ x_{q} \{ y_{l_{s+1}}, \dots, y_{l_{q}} \} \} \} \cdots \}
		, y_{r_{i+1}}, \dots, y_{r_{p}} \}
			\} \} \cdots \}}_{(\text{VII})}\\
			&=& 0.
	\end{eqnarray*}	
	The last equation holds for the following reasons:
	\begin{align*}
		& (\text{I}) + (\text{II}) = 0,\\
		& (\text{III}, i = 1) + (\text{VII}, s \neq 0) = 0,\\
		& (\text{III}, i \neq 1) + (\text{IV}, i \neq t) = 0,\\
		& (\text{IV}, i = t, t = p) + (\text{VI}) = 0,\\
		& (\text{IV}, i = t, t \neq p) + (\text{V}, q \neq j) = 0,\\
		& (\text{V}, q = j) + (\text{VII}, s = 0) = 0.
	\end{align*}
\end{proof}

We need some notations.
Define $\cals=\mathrm{End}_{\bfk s}^c$ to be the graded cooperad whose underlying graded collection is given by $\cals(n): =\Hom((\bfk s)^{\ot n}, \bfk s) \cong \bfk \delta_{n}$ for $ n\geqslant 1$, where $\delta_n \in \cals(n)$ is the map which takes $s^{\ot n}$ to $s$. The cooperad structure of $\cals$ is given by
$$\Delta_T(\delta_n):=(-1)^{(i-1)(j-1)}\delta_{n-j+1}\ot \delta_j\in \cals^{\ot T} $$
for the tree $T$ of the form
$$
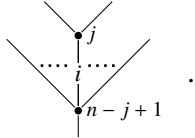
\begin{tikzpicture}[scale=1, descr/.style={fill=white}]
	\tikzstyle{every node}=[thick, minimum size=2pt, inner sep=1pt]
	\node (T) at (-1.8,0.5){};
	\node[circle, fill=black, label = right : \tiny $n-j+1$] (a) at (0,0) {};
	\node (b1) at (-1,1) {};
	\node[circle, fill=black, label = right : \tiny $j$] (b2) at (0,1) {};
	\node (b3) at (1,1) {};
	\node (c1) at (-0.5,1.5) {};
	\node (c2) at (0.5,1.5) {};
	\draw (a)--(b1);
	\draw (a)--(b2);
	\draw (a)--(b3);
	\draw (b2)--(c1);
	\draw (b2)--(c2);
	\draw [dotted, line width=1pt] (-0.5, 0.6)--(0.5, 0.6);
	\path[-,font=\scriptsize] (a) edge node[descr]{$i$} (b2);
	\node (d) at (1.5,0.4) {.};

	\node(vroot) at (0, -0.4){};
	\draw(a)--(vroot);
\end{tikzpicture}
$$
We also define $\cals^{-1}$ to be the graded cooperad whose underlying graded collection is given by $\cals^{-1}(n):=\Hom((\bfk s^{-1})^{\ot n}, s^{-1}) \cong \bfk \varepsilon_{n} $ for $n\geqslant 1$, where $\varepsilon_n\in \cals^{-1}(n)$ is the map which takes $(s^{-1})^{\ot n}$ to $s^{-1}$. The cooperad structure of $\cals^{-1}$ is given by
$$\Delta_T(\varepsilon_n):=(-1)^{(j-1)(n-j+1-i)}\varepsilon_{n-j+1}\ot \varepsilon_j\in (\cals^{-1})^{\ot T}$$
for the tree $T$ that is the same as before.

We will justify the following definition by showing its cobar construction is exactly the minimal model of $\NjAoperad$, see Definition~\ref{defn: cobar minimal model NjA} and Theorem~\ref{Thm: Minimal model NjA}, hence the name ``Koszul dual homotopy cooperad".
\begin{defn}
	The homotopy cooperad $\mathscr{S}(\NjAoperad^\ac)\ot_{\mathrm{H}} \cals^{-1}$, the Hadamard product of $\mathscr{S}(\NjAoperad^\ac)$ and $\cals^{-1}$, is called the \textit{Koszul dual homotopy cooperad of $\NjAoperad$}, denoted by $\NjAoperad^\ac$.
\end{defn}

Precisely, the underlying graded collection of ${\NjAoperad^\ac}$ is
$${\NjAoperad^\ac}(n)=\bfk e_n\oplus \bfk o_n$$
with $e_n=u_n\ot \varepsilon_n$ and $o_n=v_n\ot \varepsilon_n$ for $n\geqslant 1$, thus $|e_n|=n-1$ and $|o_n|=n$.
The defining operations $\{\Delta_T\}_{T\in \frakT}$ is given by the following formulas:
\begin{enumerate}
	\item For a tree $T$ of type $\mathrm{(I)}$,
	\begin{eqnarray*}
	\begin{tikzpicture}[scale=0.9, descr/.style={fill=white}]
	\tikzstyle{every node}=[thick, minimum size=5pt, inner sep=1pt]
	\node(v-2) at(-1.8, 1)[minimum size=0pt, label=left:{$\Delta_T(e_n)=(-1)^{(j-1)(n-i+1)}$}]{};
	\node(sigma-1) at (1.8, 0.9){.};
	\node(v0) at (0, 0)[draw, rectangle]{{\small $e_{n-j+1}$}};
	\node(v1-1) at (-1.5, 1){};
	\node(v1-2) at(0, 1)[draw, rectangle]{\small$e_j$};
	\node(v1-3) at(1.5, 1){};
	\node(v2-1)at (-1, 1.7){};
	\node(v2-2) at(1, 1.7){};
	\draw(v0)--(v1-1);
	\draw(v0)--(v1-3);
	\draw(v1-2)--(v2-1);
	\draw(v1-2)--(v2-2);
	\draw[dotted](-0.4, 1.5)--(0.4, 1.5);
	\draw[dotted](-0.5, 0.5)--(-0.1, 0.5);
	\draw[dotted](0.1, 0.5)--(0.5, 0.5);
	\path[-, font=\scriptsize]
	(v0) edge node[descr]{{\tiny$i$}} (v1-2);

	\node(vroot) at (0, -0.5){};
	\draw(v0)--(vroot);
	\end{tikzpicture}
	\end{eqnarray*}
	We consider the following three distinct cases when introducing $\Delta_T(o_n)$:
	\begin{enumerate}
		\item when $j=1$,
		$$\begin{tikzpicture}[scale=0.9, descr/.style={fill=white}]
			\tikzstyle{every node}=[thick, minimum size=5pt, inner sep=1pt]
			\node(v-1) at(-2, 0.5)[minimum size=0pt, label=left:{ $\Delta_T(o_n)=$}]{};
			\node(dd) at(1.5, 0.5){;};
			\node(v0) at (0, 0)[draw, rectangle]{\small$o_n$};
			\node(v1-1) at (-1.3, 1){};
			\node(v1-2) at(0, 1)[draw, rectangle]{\small$e_1$};
			\node(v1-3) at(1.3, 1){};
			\node(v2-1)at (0, 1.7){};
			\draw(v0)--(v1-1);
			\draw(v0)--(v1-3);
			\draw(v1-2)--(v2-1);
			\draw[dotted](-0.5, 0.5)--(-0.1, 0.5);
			\draw[dotted](0.1, 0.5)--(0.5, 0.5);
			\path[-, font=\scriptsize]
			(v0) edge node[descr]{{\tiny$i$}} (v1-2);

			\node(vroot) at (0, -0.5){};
			\draw(v0)--(vroot);
		\end{tikzpicture}$$
		\item when $2\leqslant j\leqslant n-1$, $\Delta_T(o_n)=0$;
		\item when $j=n$,
		$$\begin{tikzpicture}[scale=0.9, descr/.style={fill=white}]
			\tikzstyle{every node}=[thick, minimum size=5pt, inner sep=1pt]
			\node(va) at(2, 0.7)[minimum size=0pt, label=left:{ $\Delta_T(o_n)=$}]{};
			\node(dd) at(4.7, 0.6){.};
			\node(ve) at (3.5, 0)[draw, rectangle]{\small $e_1$};
			\begin{scope}[shift={(0,-0.3)}, scale=1]
				\node(ve1) at (3.5, 1)[draw, rectangle]{\small $o_n$};
				\node(ve2-1) at(2.5, 1.8){};
				\node(ve2-2) at(4.5, 1.8){};
				\draw[dotted](3.1, 1.5)--(3.9, 1.5);
			\end{scope}
			\draw(ve)--(ve1);
			\draw(ve1)--(ve2-1);
			\draw(ve1)--(ve2-2);

			\node(vroot) at (3.5, -0.5){};
			\draw(ve)--(vroot);
		\end{tikzpicture}$$
	\end{enumerate}

	\item For a tree $T$ of type $\mathrm{(II)}$,

	\begin{eqnarray*}
		\begin{tikzpicture}[scale=1.2, descr/.style={fill=white}]
			\tikzstyle{every node}=[thick, minimum size=3pt, inner sep=1pt]
	
			\node(v-0-1) at(-4.3, 1.3)[minimum size=0pt, label=right:{$\Delta_T(o_n)=(-1)^\gamma$}]{};
			\node(v-ddd) at(2.5, 1.2){,};
	
			\begin{scope}[shift={(0,0.8)}, scale=1]
				\node(v-1) at (0, -2.2)[rectangle, draw]{\tiny $o_{r_{1}}$};
				\node(v0-1) at (-2, -1){ };
				\node(v0-3) at(2, -1){ };	
							
				\node(v-1-2) at (0, -1.6)[rectangle, draw]{\tiny $o_{r_{2}}$};
				\node(v0-1-2) at (-2, -0.4){ };
				\node(v0-3-2) at(2, -0.4){ };
				\node(v0-2-1-2) at(0, -0.9)[]{};
				\node(v0-2-2) at(0, -0.6)[]{};
				\draw(v-1-2)--(v0-1-2);
				\draw(v-1-2)--(v0-2-1-2);
				\draw(v-1-2)--(v0-3-2);
				\draw(v-1)--(v-1-2);
				\draw[dotted](v0-2-1-2)--(v0-2-2);
				\path[-, font=\scriptsize]
		
				(v-1) edge node[descr]{{\tiny$i_{1}$}} (v-1-2);
				\path[-, font=\scriptsize]
				(v-1-2) edge node[descr]{{\tiny$i_{2}$}} (v0-2-1-2);
		
				\draw[dotted](-0.5, -1.8)--(-0.1, -1.8);
				\draw[dotted](0.1, -1.8)--(0.5, -1.8);
		
				\draw[dotted](-0.5, -1.2)--(-0.1, -1.2);
				\draw[dotted](0.1, -1.2)--(0.5, -1.2);

				\node(vroot) at (0, -2.6){};
				\draw(v-1)--(vroot);
			\end{scope}
				\node(v0) at (0, 0.5)[rectangle, draw]{\tiny $o_{r_{t}}$};
				\node(v1-1) at (-2, 1.8){ };
				\node(v1-2) at(0, 1.2)[rectangle, draw]{\tiny $e_{p}$};
				\draw[dotted](-0.6, 1)--(-0.1, 1);
				\draw[dotted](0.1, 1)--(0.6, 1);
		
				\node(v1-3) at(2, 1.8){ };
				\node(v2-1) at(-1.9, 2.6){ };
				\node(v2-2) at (-0.9, 2.8)[rectangle, draw]{\tiny $o_{r_{t+1}}$};
				\node(v2-3) at (0, 2.9){};
				\node(v2-4) at(0.9, 2.8)[rectangle, draw]{\tiny $o_{r_{p}}$};
				\node(v2-5) at(1.9, 2.6){ };
				\node(v3-1) at (-1.6, 3.4){ };
				\node(v3-2) at (-0.4, 3.4){ };
				\node(v3-3) at (0.4, 3.4){ };
				\node(v3-4) at(1.6, 3.4){ };
		
				\draw(v0-2-2)--(v0);
				\draw(v0)--(v1-1);
				\draw(v0)--(v1-3);
				\draw(v-1)--(v0-1);
				\draw(v-1)--(v0-3);
				
				\path[-, font=\scriptsize]
				(v0) edge node[descr]{{\tiny$i_{t}$}} (v1-2);
				
				\draw(v1-2)--(v2-1);
				\draw(v1-2)--(v2-3);
				\draw(v1-2)--(v2-5);
				
				\path[-, font=\scriptsize]
				(v1-2) 	edge node[descr]{{\tiny$k_t$}} (v2-2)
				edge node[descr]{{\tiny$k_{p-1}$}} (v2-4);
				\draw(v2-2)--(v3-1);
				\draw(v2-2)--(v3-2);
				\draw(v2-4)--(v3-3);
				\draw(v2-4)--(v3-4);
		
				\draw[dotted](-0.5, 2.4)--(-0.1, 2.4);
				\draw[dotted](0.1, 2.4)--(0.5, 2.4);
				\draw[dotted](-1.4, 2.4)--(-0.8, 2.4);
				\draw[dotted](1.4, 2.4)--(0.8, 2.4);
				\draw[dotted](-1.2, 3.2)--(-0.7, 3.2);
				\draw[dotted](1.2, 3.2)--(0.7, 3.2);
			\end{tikzpicture}
	\end{eqnarray*}
	where
	\begin{align*}
	\gamma
	= \sum_{h=1}^{t} \Big(1-n-h+ \sum_{j=1}^{h} r_{j} \Big) (r_{h} - i_{h})
	+ \sum_{j=t+1}^{p}(r_{j}-1) (p-k_{j-1})
	+ \sum_{j=1}^{p-1}(p-j)r_{j}
	+ (p-t)(p-1).
	\end{align*}
	\item For any other tree $T$ in $\frakT$, $\Delta_T = 0$.
\end{enumerate}

\medskip

\section{The Minimal model of the operad for Nijenhuis associative algebras} \label{Section: minimal model NjA}

 In this section, we will prove that the cobar construction of the homotopy cooperad $\NjAoperad^\ac$ is exactly the minimal model for $\NjAoperad$, the operad of Nijenhuis associative algebras. Then we will introduce the notion of homotopy Nijenhuis associative algebras. 

\subsection{The dg operad for homotopy Nijenhuis associative algebras}\

 \begin{defn} \label{defn: cobar minimal model NjA}
	 The \textit{dg operad for homotopy Nijenhuis associative algebras}, denoted by $\NjAoperad_{\infty}$, is defined to be the cobar construction $\Omega(\NjAoperad^\ac)$.
\end{defn}

By the definition of the cobar construction of a coaugmented homotopy cooperad, see \cite[Definition 3.4]{CGWZ24}, the dg operad $\NjAoperad_{\infty}=\Omega(\NjAoperad^\ac)$ is the free operad generated by the graded collection $s^{-1} \overline{\NjAoperad^\ac}$ endowed with differential induced from the homotopy cooperad structure on $\NjAoperad^\ac$. More precisely,
$$ s^{-1} \overline{\NjAoperad^\ac}(1) = \bfk s^{-1}o_{1} \quad \mathrm{and} \quad s^{-1} \overline{\NjAoperad^\ac}(n) = \bfk s^{-1}e_{n} \oplus \bfk s^{-1}o_{n},\ n \geqslant 2. $$
Denote $m_{n} = s^{-1}e_{n}$ for $n \geqslant 2$ and $P_{n} = s^{-1} o_{n}$ for $n \geqslant 1$ respectively, so $|m_{n}| = n-2$ and $|P_{n}| = n-1$. A direct calculation shows that the action of differential on these generators in $\Omega(\NjAoperad^\ac)$ is given by:
for $ n \geqslant 2 $,
\begin{equation} \label{Eq. homotopy NjA eq1}
	\begin{aligned}
		\partial(m_n)
		 =&\ \sum_{j=2}^{n-1}\sum_{i=1}^{n-j+1}(-1)^{i+j(n-i)}m_{n-j+1}\circ_i m_j,
	\end{aligned}
\end{equation}
and for $ n \geqslant 1 $,
\begin{equation} \label{Eq. homotopy NjA eq2}
	\begin{aligned}
			 \partial (P_{n})
			= &\ \sum_{ \substack{ r_1+\cdots+r_p =n
					\\	r_1, \dots, r_p \geqslant 1
					\\	2 \leqslant p \leqslant n} }
			\sum_{ t=0}^{p}
			\sum_{\substack{ 1 \leqslant i_{1} \leqslant r_{1} \\ \cdots \\ 1 \leqslant i_{t} \leqslant r_{t} }}
			\sum_{1 \leqslant k_{t} < \cdots < k_{p-1} \leqslant p}
			(-1)^{\alpha'} \\
			& P_{r_{1}} \circ_{i_{1}}
			\Bigg(
			P_{r_{2}} \circ_{i_{2}}
			\bigg(
			\cdots
			\circ_{i_{t-1}}
			\Big(
			P_{r_{t}} \circ_{i_{t}}
			\big(
			(
			\cdots
			((m_{p} \circ_{k_{t}} P_{r_{t+1}}) \circ_{k_{t+1}+r_{t+1}-1} P_{r_{t+2}} )
			\cdots
			) \circ_{\beta} P_{r_{p}}
			\big)
			\Big)
			\cdots
			\bigg)
			\Bigg),
	\end{aligned}
\end{equation}
where	
\begin{align*}
	\beta &= k_{p-1}+r_{t+1} + \cdots + r_{p-1}-(p-1-t),\\
	\alpha'
	&= 1+ \sum^{t}_{q=1} \Big(i_{q} + \big(\sum_{s=q+1}^{p} r_{s}\big) (r_{q}-i_{q}) - q(r_{q}-i_{q})\Big)
	+ \sum_{i=t+1}^{p}(k_{i-1} -p)(r_{i}-1).
\end{align*}

%


We use the corolla with $n$ leaves and a solid vertex to represent generators $m_n$ for $ n\geqslant 2$ and the corolla with $n$ leaves and a hollow vertex to represent generators $P_n$ for $n\geqslant 1$:
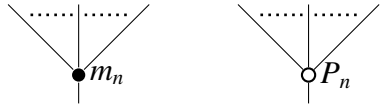
\begin{eqnarray*}
	\begin{tikzpicture}[scale=0.8]
		\tikzstyle{every node}=[thick, minimum size=5pt, inner sep=1pt]
		\node(r)[fill=black, circle, label=right:$m_n$] at (0,0){};
		\node (a1) at (-1.3,1.3){}; 
			\draw (r)--(a1);
		\node (a2) at (0,1.3){};
			\draw (r)--(a2);
		\node (a2) at (1.3,1.3){}; 
			\draw (r)--(a2);
		\draw [dotted, line width=1pt] (-0.8, 1)--(-0.1, 1);
		\draw [dotted, line width=1pt] (0.8, 1)--(0.1, 1);

		\node(vroot) at (0, -0.6){};
		\draw(r)--(vroot);
	\end{tikzpicture}
	\hspace{8mm}
	\begin{tikzpicture}[scale=0.8]
		\tikzstyle{every node}=[thick, minimum size=5pt, inner sep=1pt]
		\node(r)[draw, circle, label=right: $P_n$] at (0,0){};
		\node (a1) at (-1.3,1.3){}; 
			\draw (r)--(a1);
		\node (a2) at (0,1.3){};
			\draw (r)--(a2);
		\node (a2) at (1.3,1.3){}; 
			\draw (r)--(a2);
		\draw [dotted, line width=1pt] (-0.8, 1)--(-0.1, 1);
		\draw [dotted, line width=1pt] (0.8, 1)--(0.1, 1);

		\node(vroot) at (0, -0.6){};
		\draw(r)--(vroot);
	\end{tikzpicture}
\end{eqnarray*}

A tree with all vertices dyed solid or hollow (such a tree will be called a tree monomial) gives an element in $\NjAoperad_\infty$ by composing its vertices in the planar order, which is given by counting the vertices starting from the root clockwisely along the tree.
Conversely, any element in $\NjAoperad_\infty$ can be represented by such a tree monomial in this way.
In this means, the action of the differential operator $\partial$ on generators, i.e., Equations~\eqref{Eq. homotopy NjA eq1}-\eqref{Eq. homotopy NjA eq2}, can be expressed by trees as follows:
\begin{eqnarray*}
&\begin{tikzpicture}
	\tikzstyle{every node}=[thick, minimum size=5pt, inner sep=1pt]
	\begin{scope}[shift={(2,0)}] 
		\node(a) at (-9.7, 0){\large $\partial$};
		\node(bd) at (0,-0.2){,};
		\node[circle, fill=black, label=right:$m_n$] (b0) at (-8.5, -0.45) {};
		\node (b1) at (-9.3, 0.5) [minimum size=0pt, label=above:]{}; 
		\node (b2) at (-8.5, 0.5) [minimum size=0pt]{};
		\node (b3) at (-7.7, 0.5) [minimum size=0pt, label=above:]{}; 
		\draw (b0)--(b1);
		\draw (b0)--(b2);
		\draw (b0)--(b3);
		\draw [dotted, line width=1pt] (-9, 0.3)--(-8.6, 0.3);
		\draw [dotted, line width=1pt] (-8.4, 0.3)--(-8, 0.3);

		\node(vroot) at (-8.5, -0.9){};
		\draw(b0)--(vroot);
	\end{scope}
	\node(eq) at (-3.5, 0){$= \sum\limits_{j=2}^{n-1} \sum\limits_{i=1}^{n-j+1}
		(-1)^{i+j(n-i)} $};
	\node(e0) at (0, -1.5)[circle, fill=black, label=right:$m_{n-j+1}$]{};
	\node(e1) at(-1.5, 0){}; 
	\node(e2-0) at (0, -0.5){{\tiny$i$}}; %
	\node(e3) at (1.5, 0){}; 
	\node(e2-1) at (0, 0.5) [circle, fill=black, label=right: $m_j$]{};
	\node(e2-1-1) at (-1, 1.5){}; 
	\node(e2-1-2) at (1, 1.5){}; 
	\draw [dotted, line width=1pt] (-0.7, -0.5)--(-0.2, -0.5);
	\draw [dotted, line width=1pt] (0.3, -0.5)--(0.8, -0.5);
	\draw [dotted, line width=1pt] (-0.4, 1.2)--(0.4, 1.2);
	\draw (e0)--(e1);
	\draw (e0)--(e3);
	\draw (e0)--(e2-0);
	\draw (e2-0)--(e2-1);
	\draw (e2-1)--(e2-1-1);
	\draw (e2-1)--(e2-1-2);

	\node(vroot) at (0, -2){};
	\draw(e0)--(vroot);
\end{tikzpicture}
\\
&\begin{tikzpicture}[scale=0.6]
\tikzstyle{every node}=[thick, minimum size=5pt, inner sep=1pt]
	\begin{scope}[shift={(1.5,0)}] 
		\node(a) at (-4, 0){\large$\partial$};
		\node(bd) at (22.5,-0.3){.};
		\node[circle, draw, label=right:$P_n$] (b0) at (-2, -0.7) {};
		\node (b1) at (-3.5, 1) [minimum size=0pt]{}; 
		\node (b2) at (-2, 1) [minimum size=0pt]{};
		\node (b3) at (-0.5, 1) [minimum size=0pt]{}; 
		\draw (b0)--(b1);
		\draw (b0)--(b2);
		\draw (b0)--(b3);
		\draw [dotted, line width=1pt] (-2.9, 0.6)--(-2.2, 0.6);
		\draw [dotted, line width=1pt] (-1.8, 0.6)--(-1.1, 0.6);

		\node(vroot) at (-2, -1.5){};
		\draw(b0)--(vroot);
	\end{scope}
	\node (eq2) at (7, -0.5){$= \sum\limits_{ \substack{ r_1+\cdots+r_p =n
				\\	r_1, \dots, r_p \geqslant 1
				\\	2 \leqslant p \leqslant n} }
		\sum\limits_{ t=0}^{p}
		\sum\limits_{\substack{ 1 \leqslant i_{1} \leqslant r_{1} \\ \cdots \\ 1 \leqslant i_{t} \leqslant r_{t} }}
		\sum\limits_{1 \leqslant k_{t} < \cdots < k_{p-1} \leqslant p}
		(-1)^{\alpha'} $};
	\begin{scope}[shift={(18,-5)}, scale=1.3]
		\node(1) at(0, 0)[circle, draw ]{};
		\node(1+) at(0.4, -0.4){\tiny $ P_{r_{1}} $};
		\node(1-1) at (-4, 1.5){};
		\node(1-2) at (0, 0.9){\tiny $ i_{1} $};
		\node(1-3) at (4, 1.5){};
		\draw(1)--(1-1);
		\draw(1)--(1-2);
		\draw(1)--(1-3);
		\draw[dotted, line width=1pt](-1, 0.6)--(-0.2, 0.6);
		\draw[dotted, line width=1pt](0.2, 0.6)--(1, 0.6);
		
		\node(2) at(0, 1.5)[circle, draw ]{};
		\node(2+) at(0.4, 1.1){\tiny $ P_{r_{2}} $};
		\node(2-1) at (-4, 3){};
		\node(2-2) at (0, 2.4){\tiny $ i_{2} $};
		\node(2-3) at (4, 3){};
		\draw(2)--(2-1);
		\draw(2)--(2-2);
		\draw(2)--(2-3);
		\draw(1-2)--(2);
		\draw[dotted, line width=1pt](-1, 2.1)--(-0.2, 2.1);
		\draw[dotted, line width=1pt](0.2, 2.1)--(1, 2.1);
		
		\node(3) at (0, 4)[circle, draw ]{};
		\node(3+) at(0.4, 3.6){\tiny $ P_{r_{t}} $};
		\node(3-) at (0, 3.45){};
		\draw[dotted, line width=1pt](2-2)--(3-);
		\draw(3-)--(3);
		\node(3-1) at (-4, 5.5){};
		\node(3-2) at (0, 4.9){\tiny $ i_{t} $};
		\node(3-3) at (4, 5.5){};
		\draw(3)--(3-1);
		\draw(3)--(3-2);
		\draw(3)--(3-3);
		\draw[dotted, line width=1pt](-1, 4.6)--(-0.2, 4.6);
		\draw[dotted, line width=1pt](0.2, 4.6)--(1, 4.6);
		
		\node(4) at (0, 5.5)[circle, fill=black]{};
		\draw(3-2)--(4);
		\node(3+) at(0.4, 5.1){\tiny $ m_{p} $};
		\node(4-0) at (-4, 7){};
		\node(4-1) at (-2.5, 7)[circle, draw, label=left:\tiny$ P_{r_{t+1}} $]{};
		\node(4-2) at (-0.5, 7)[circle, draw, label=left:\tiny$ P_{r_{t+2}} $]{};
		\node(4-3) at (2.5, 7)[circle, draw, label=left:\tiny$ P_{r_{p}} $]{};
		\node(4-4) at (4, 7){};
		
		\node(4-1-) at (-1.48, 6.4){\tiny$ k_{t} $};
		\node(4-2-) at (-0.3, 6.4){\tiny$ k_{t+1} $};
		\node(4-3-) at (1.5, 6.4){\tiny$ k_{p-1} $};
		
		\draw(4)--(4-0);
		\draw(4)--(4-1-);
		\draw(4)--(4-2-);
		\draw(4)--(4-3-);
		\draw(4)--(4-4);
		
		\draw(4-1-)--(4-1);
		\draw(4-2-)--(4-2);
		\draw(4-3-)--(4-3);
		\draw[dotted, line width=1pt](-1.2, 6)--(-0.9, 6);
		\draw[dotted, line width=1pt](-0.7, 6)--(-0.2, 6);
		\draw[dotted, line width=1pt](-0.1, 6)--(0.8, 6);
		\draw[dotted, line width=1pt](0.9, 6)--(1.2, 6);
		
		\node(4-1-1) at (-3, 8){};
		\node(4-1-2) at (-2, 8){};
		\draw(4-1)--(4-1-1);
		\draw(4-1)--(4-1-2);
		\draw[dotted, line width=1pt](-2.8,7.7)--(-2.2,7.7);
		
		\node(4-2-1) at (-1, 8){};
		\node(4-2-2) at (0, 8){};
		\draw(4-2)--(4-2-1);
		\draw(4-2)--(4-2-2);
		\draw[dotted, line width=1pt](-0.8,7.7)--(-0.2,7.7);
		
		\node(4-3-1) at (2, 8){};
		\node(4-3-2) at (3, 8){};
		\draw(4-3)--(4-3-1);
		\draw(4-3)--(4-3-2);
		\draw[dotted, line width=1pt](2.2,7.7)--(2.8,7.7);

		\node(vroot) at (0, -0.8){};
		\draw(1)--(vroot);
	\end{scope}
\end{tikzpicture}
\end{eqnarray*}

\subsection{The minimal model}\

Let's recall the notion of minimal models of operads.
For a graded collection $M=\{M(n)\}_{n\geqslant 1} $, denote by $ \mathcal{F}(M)$ the free graded operad generated by $M$. Recall that a dg operad is called quasi-free if its underlying graded operad is free.
\begin{defn} \cite{DCV13}\label{minimal model}
	A \textit{minimal model} of a dg operad $\mathcal{P}$ is a quasi-free dg operad $ (\mathcal{F}(M), \partial)$ together with a surjective quasi-isomorphism of dg operads $(\mathcal{F}(M), \partial)\overset{\sim}{\twoheadrightarrow}\mathcal{P}$, where the dg operad $(\mathcal{F}(M), \partial)$ satisfies the following conditions:
	\begin{enumerate}
		\item the differential $\partial$ is decomposable, i.e., $\partial$ takes $M$ to $\mathcal{F}(M)^{(\geqslant 2)}$, the subspace of $\mathcal{F}(M)$ consisting of elements with weight $\geqslant 2$; \label{it:min1}
		\item the generating collection $M$ admits a decomposition $M=\bigoplus\limits_{i\geqslant 1}M_{(i)}$ such that $\partial(M_{(k+1)})\subset \mathcal{F}\big(\bigoplus\limits_{i=1}^kM_{(i)}\big)$ for all $k\geqslant 1$. \label{it:min2}
	\end{enumerate}
\end{defn}
\begin{thm} \cite{DCV13}
	When a dg operad $\mathcal{P}$ admits a minimal model, it is unique up to isomorphisms.
\end{thm}

The following result is the main result of this section. 

\begin{thm} \label{Thm: Minimal model NjA}
	The dg operad $\NjAoperad_\infty$ is the minimal model of the operad $\NjAoperad$.
\end{thm}

\begin{proof}

 We follow the line of proof for the minimal model of Rota-Baxter associative operad \cite{WZ24}. However, there are important differences.

It can be easily seen that the differential $\partial$ on $\NjAoperad_\infty$ satisfies Conditions (i) and (ii) in Definition~\ref{minimal model}, the definition of minimal models.
In order to prove Theorem~\ref{Thm: Minimal model NjA}, we only need to construct a surjective quasi-isomorphism of dg operads from $\NjAoperad_\infty$ to $\NjAoperad$, where $\NjAoperad$, defined in Definition~\ref{def: operad NjA}, is considered as a dg operad concentrated in degree 0.

%

Introduce a morphism of dg operads $\phi: \NjAoperad_\infty \twoheadrightarrow \NjAoperad$, which sends generators $m_{2}$ to $m$, $P_{1}$ to $P$, and all other generators to $0$.
	The degree zero part of $\NjAoperad_\infty$ is the free graded operad generated by $ m_2$ and $P_1$. The image of $\partial$ in this degree zero part is the operadic ideal generated by $\partial (m_3)$ and $\partial (P_2) $. By definition, we have:
	\begin{align*}
		\partial(m_3) &= - m_2 \circ_1 m_2 + m_2 \circ_2 m_2, \\
		\partial (P_2)
		&= - (m_{2} \circ_{1} P_{1}) \circ_{2} P_{1}
		+ P_{1} \circ_{1} (m_{2} \circ_{1} P_{1})
		+ P_{1} \circ_{1} (m_{2} \circ_{2} P_{1})
		- (P_{1} \circ_{1} P_{1}) \circ_{1} m_{2}.
	\end{align*}
	Thus the map $\phi$ induces the isomorphism $\rmH_0(\NjAoperad_\infty, \partial) \cong \NjAoperad$.

To prove the map $\phi: \NjAoperad_\infty \twoheadrightarrow \NjAoperad$ is a quasi-isomorphism, we just need to prove that $\rmH_i(\NjAoperad_\infty, \partial)=0$ for all $i\geqslant 1$. This will be achieved by constructing a homotopy map. To this end, we need to establish an (arity-graded) monomial order on $\NjAoperad_{\infty}$. The main difference from \cite{WZ24} is that we need a different monomial order in order that the leading terms are the same as the case of Rota-Baxter associative operad.

By \cite[Chapter 3]{BD16}, each tree monomial $\mathcal{T}$ in $ \NjAoperad_{\infty} $ with $ n $ leaves can be represented by $\mathcal{T} = (X_{1}, X_{2}, \dots, X_{n})$, where $ X_{i} $ is the word formed by generators of $ \NjAoperad_{\infty} $ corresponding to the vertices along the unique path from the root of $ \mathcal{T} $ to its $i$-th leaf counted from left to right, for all $1 \leqslant i \leqslant n$.

Define a function $ \varphi $, called weight, on generators of $\NjAoperad_{\infty}$, as follows:
$$\begin{tikzpicture}[scale=0.5]
	\tikzstyle{every node}=[thick, minimum size=5pt, inner sep=1pt]
	\node(a) at (-4, 0.5){\large$\varphi$};
	\node[circle, fill=black, label=right:$m_n$] (b0) at (-2, -0.4) {};
	\node (b1) at (-3.5, 1.5) [minimum size=0pt]{}; 
	\node (b2) at (-2, 1.5) [minimum size=0pt]{};
	\node (b3) at (-0.5, 1.5) [minimum size=0pt]{}; 
	\draw (b0)--(b1);
	\draw (b0)--(b2);
	\draw (b0)--(b3);
	\draw [dotted, line width=1pt] (-3, 1)--(-2.2, 1);
	\draw [dotted, line width=1pt] (-1.8, 1)--(-1, 1);
	\node(0) at (3,0.5){$:=n-1(n\geqslant2)$,};

	\node(vroot) at (-2, -1.4){};
	\draw(b0)--(vroot);
\end{tikzpicture}
\quad
\begin{tikzpicture}[scale=0.5]
	\tikzstyle{every node}=[thick, minimum size=5pt, inner sep=1pt]
	\node(a) at (-4, 0.5){\large$\varphi$};
	\node[circle, draw, label=right:$P_n$] (b0) at (-2, -0.4) {};
	\node (b1) at (-3.5, 1.5) [minimum size=0pt]{}; 
	\node (b2) at (-2, 1.5) [minimum size=0pt]{};
	\node (b3) at (-0.5, 1.5) [minimum size=0pt]{}; 
	\draw (b0)--(b1);
	\draw (b0)--(b2);
	\draw (b0)--(b3);
	\draw [dotted, line width=1pt] (-3, 1)--(-2.2, 1);
	\draw [dotted, line width=1pt] (-1.8, 1)--(-1, 1);
	\node(0) at (3,0.5){{$:=2n-1(n\geqslant1)$.}};

	\node(vroot) at (-2, -1.4){};
	\draw(b0)--(vroot);
\end{tikzpicture}
$$

Now, we establish a monomial order $ \Xi $ as we need. For two tree monomials $\mathcal{T} = (X_1, \dots, X_n)$, $\mathcal{T}' = (Y_1, \dots, Y_m)$, we compare $\mathcal{T}, \mathcal{T}'$ in the following way:
\begin{enumerate}
	\item Compare arity. If $n > m$, then $\mathcal{T}>\mathcal{T}'$; 
	\item Compare the first path. If $n = m$, we compare $\mathcal{T}, \mathcal{T}'$ as follows: 
	Let $ X_{1}=x_{1} x_{2} \cdots x_{r} $ and $ Y_{1}= y_{1} y_{2} \cdots y_{s} $ be the paths from the root of $ \mathcal{T} $ and $ \mathcal{T}' $, respectively, to the first leaf.
	\begin{enumerate}
		\item Compare $ \varphi(X_{1}):= \varphi(x_{1}) + \cdots + \varphi(x_{r}) $ and $ \varphi(Y_{1}):= \varphi(y_{1}) + \cdots + \varphi(y_{s}) $, that is, if $ \varphi(X_{1}) > \varphi(Y_{1}) $, then $ \mathcal{T} > \mathcal{T}' $;
		\item If $ \varphi(X_{1}) = \varphi(Y_{1}) $, then compare the following words:
		\begin{eqnarray*}
			\underbrace{x_{1} \cdots x_{1}}_{\varphi(x_{1})} \underbrace{x_{2} \cdots x_{2}}_{\varphi(x_{2})} \cdots \underbrace{x_{r} \cdots x_{r}}_{\varphi(x_{r})}
			\quad \mathrm{and} \quad
			\underbrace{y_{1} \cdots y_{1}}_{\varphi(y_{1})} \underbrace{y_{2} \cdots y_{2}}_{\varphi(y_{2})} \cdots \underbrace{y_{s} \cdots y_{s}}_{\varphi(y_{s})},
		\end{eqnarray*}
		that is if the former is greater than the latter with respect to the lexicographic order induced by the following total order:
		$$ P_{1} < m_{2} < P_{2} < m_{3} < \cdots < P_{n} < m_{n+1} < P_{n+1} < \cdots, $$
		then $ \mathcal{T} > \mathcal{T}' $.
	\end{enumerate}
	\item Compare the other paths. If we can't compare $ \mathcal{T} $ and $ \mathcal{T}' $ using paths ending at first leaves, then repeat the process, outlined in step (ii), for their leaves $ 2, 3, \dots $, in turn until the size is distinguished.
\end{enumerate}

It is ready to see that this is indeed a monomial order.
Under this order, the leading term in the expansion of $\partial(m_n)$, $\partial(P_n)$ are the following tree monomials respectively:
$$
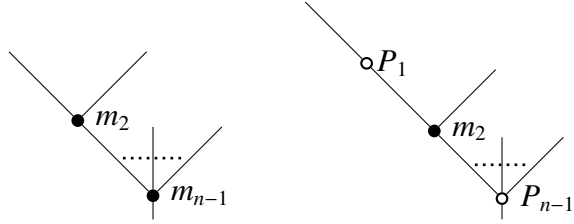
\begin{tikzpicture}
		\tikzstyle{every node}=[thick, minimum size=4pt, inner sep=1pt]
		\node(1) at (0, 0) [draw, circle, fill=black, label=right:$\ m_{n-1}$]{};
		\node(2-1) at (-1, 1) [draw, circle, fill=black, label=right: $\ m_2$]{};
		\node(3-1) at (-2, 2){};
		\node(3-2) at (0, 2){};
		\node(2-2) at (0, 1){};
		\node(2-3) at (1, 1){};
		\draw (1)--(2-1);
		\draw (1)--(2-2);
		\draw (1)--(2-3);
		\draw (2-1)--(3-1);
		\draw (2-1)--(3-2);
		\draw [dotted, line width=1pt](-0.4, 0.5)--(0.4, 0.5);

		\node(vroot) at (0,-0.4){};
		\draw(1)--(vroot);
	\end{tikzpicture}
	\hspace{8mm}
	\begin{tikzpicture}[scale=0.9]
		\tikzstyle{every node}=[thick, minimum size=4pt, inner sep=1pt]
		\node(1) at (0, 0) [draw, circle, label=right:$\ P_{n-1}$]{};
		\node(2-1) at (-1, 1) [draw, circle, fill=black, label=right: $\ m_2$]{};
		\node(3-1) at (-2, 2)[draw, circle, label=right: $P_1$]{};
		\node(3-2) at (0, 2){};
		\node(2-2) at (0, 1){};
		\node(2-3) at (1, 1){};
		\node(4) at (-3, 3){};
		\draw (1)--(2-1);
		\draw (1)--(2-2);
		\draw (1)--(2-3);
		\draw (2-1)--(3-1);
		\draw (2-1)--(3-2);
		\draw (3-1)--(4);
		\draw [dotted, line width=1pt](-0.4, 0.5)--(0.4, 0.5);

		\node(vroot) at (0,-0.4){};
		\draw(1)--(vroot);
	\end{tikzpicture}
$$

Once the leading terms are seen to be the same as the case of Rota-Baxter associative operad, the remaining part of the proof carries verbatim as that of \cite[Theorem 3.5]{WZ24}.

\end{proof}

\subsection{Homotopy Nijenhuis associative algebras}\

We can now introduce the notion of homotopy Nijenhuis associative algebras.

\begin{defn} \label{defn. homotopy Nijenhuis associative algebra via dg operad}
	Let $(V,d_V)$ be a complex. Then a \textit{homotopy Nijenhuis associative algebra} (or a \textit{$\NjAoperad_\infty$-algebra}) on $V$ is defined to be a morphism of dg operads from $\NjAoperad_\infty$ to the endomorphism dg operad $\End_V$.
\end{defn}

Let $(V, d_V)$ be an algebra over the dg operad $\NjAoperad_\infty$.
Still denote by
$ m_n: V^{\ot n}\rightarrow V , n\geqslant 2$ (resp. $ P_n: V^{\ot n}\rightarrow V, n\geqslant 1$) the image of $m_n\in \NjAoperad_\infty$ (resp. $P_n\in \NjAoperad_\infty$).
We also rewrite $m_1 := d_V$.
For each $n\geqslant 1$, Equations~\eqref{Eq. homotopy NjA eq1}-\eqref{Eq. homotopy NjA eq2} are respectively equivalent to:
\begin{eqnarray}\label{Eq: stasheff-id NjA}
	\sum_{ i+j+k= n, \atop
		i, k\geqslant 0, j\geqslant 1 } (-1)^{i+jk}m_{i+1+k}\circ (\id^{\ot i}\ot m_j\ot \id^{\ot k} )=0
\end{eqnarray}
and
\begin{equation} \label{Eq: homotopy NjA-operator-version-2 NjA}
	\begin{aligned}
		\sum_{ r_1+\cdots+r_p =n
			\atop
			r_1, \dots, r_p \geqslant 1, p \geqslant 1 }
		\sum_{ t=0}^{p}
		\sum_{\substack{i_{1}+k_{1}+1=r_{1} \\ \cdots \\ i_{t}+k_{t}+1=r_{t} \\ i_{1}, \dots, i_{t} \geqslant 0 \\ k_{1}, \dots, k_{t} \geqslant 0 }}
		\sum_{\substack{h_{t}+ \cdots + h_{p} =t \\ h_{t}, \dots, h_{p} \geqslant 0}}
		(-1)^{\alpha} &\\
		P_{r_{1}} \circ \bigg( \id^{\ot i_{1}} \ot P_{r_{2}} \circ \Big( \id^{\ot i_{2}} \ot \cdots \ot P_{r_{t}} \circ \big(\id^{\ot i_{t}} \ot & \\
		m_{p} \circ (\id^{h_{t}} \ot P_{r_{t+1}} \ot \id^{\ot h_{t+1}} \ot P_{r_{t+2}} \ot \cdots \ot P_{r_{p-1}} \ot &\ \id^{h_{p-1}} \ot P_{r_{p}} \ot \id^{\ot h_{p}}) \\
		\ot &\ \id^{k_{t}}\big) \cdots \ot \id^{\ot k_{2}}\Big)\ot \id^{\ot k_{1}}\bigg) =0,
	\end{aligned}
\end{equation}
where
\begin{eqnarray*}
	\alpha
	= \sum^{t}_{q=1} \Big(r_{q} + k_{q} + \big(\sum_{s=q+1}^{p} r_{s}\big)k_{q} - qk_{q}\Big)
	+ \sum_{i=t+1}^{p}\big(\sum_{j=t}^{i-1} h_{j} + i - t -p \big) (r_{i}-1).
\end{eqnarray*}

Equation~(\ref{Eq: stasheff-id NjA}) is exactly the Stasheff identity in the definition of $A_\infty$-algebras \cite{Sta63}. In particular, the operator $m_1$ is a differential on $V$, and the operator $m_2$ is associative up to homotopy, in other words, $m_2$ induces an associative algebra structure on the homology $\rmH_\bullet(V, m_1)$.

For $n=1, 2$, Equation~(\ref{Eq: homotopy NjA-operator-version-2 NjA}) gives
\begin{eqnarray} \label{operator-differential NjA}
	m_1\circ P_1=P_1\circ m_1
\end{eqnarray}
and
\begin{align} \label{rbo-homotopy NjA}
	m_2\circ(P_1\ot P_1) & -P_1\circ m_2\circ (\id\ot P_1)-P_1\circ m_2\circ (P_1\ot \id) + P_1\circ P_1\circ m_2 \\
	\notag =& - (m_1\circ P_2+P_2\circ (\id\ot m_1)+P_2\circ(m_1\ot \id)).
\end{align}
Equation~(\ref{operator-differential NjA}) implies that $P_1: (V, m_1)\rightarrow (V, m_1)$ is a chain map, thus $P_1$ is well-defined on $\rmH_\bullet(V, m_1)$; Equation~(\ref{rbo-homotopy NjA}) indicates that $P_1$ is a Nijenhuis operator with respect to $m_2$ up to homotopy, whose obstruction is exactly the operator $P_2$. As a consequence, $(\rmH_\bullet(V, m_1), m_2, P_1)$ is a Nijenhuis associative algebra.

\medskip

\medskip

\section{From the minimal model to the $L_\infty$-algebra structure on deformation complex} \label{Section: From minimal model to Linifnity algebras NjA}



In this section, we will use the minimal model $\NjAoperad_\infty$, or more precisely the Koszul dual homotopy cooperad ${\NjAoperad^{\ac}}$, to determine the deformation complex as well as the $L_\infty$-algebra structure on it for Nijenhuis associative algebras. 

\subsection{The $L_\infty$-algebra structure on deformation complex}\

Given a homotopy cooperad $\calc$ and a dg operad $\calp$, one can construct a homotopy operad $\mathbf{Hom}(\calc, \calp)$, known as the convolution homotopy operad. This construction naturally induces an $L_\infty$-algebra $\mathbf{Hom}(\calc, \calp)^{\prod}$. For precise definitions and details of the construction, refer to Proposition~3.5 and Proposition-Definition~3.8 in \cite{CGWZ24}. Hence, we obtain the following $L_\infty$-algebra:
\begin{defn}
	Let $V$ be a graded space. Introduce an $L_\infty$-algebra associated to $V$ as $\frakC_{\NjA}(V):=\mathbf{Hom}({\NjAoperad^\ac}, \End_V)^{\prod}$.
\end{defn}

Now, let's determine the $L_\infty$-algebra $\frakC_{\NjA}(V)$ explicitly. The sign rules in the homotopy cooperad ${\NjAoperad^\ac}$ are complicated, so we need some transformations. Notice that there is a natural isomorphism of graded operads $\mathbf{Hom}(\cals, \End_{sV})\cong \End_V$.
Explicitly, any element $h \in \End_V(n)$ corresponds to an element $\hbar \in\mathbf{Hom}(\cals, \End_{sV})(n)$ which is defined as
$$\big(\hbar(\delta_n)\big)(sv_1\ot\cdots\ot sv_n) := (-1)^{\sum\limits_{k=1}^{n-1}\sum\limits_{j=1}^k|v_j|}(-1)^{(n-1)|h|}sh(v_1\ot \cdots \ot v_n)$$
for any homogeneous elements $v_1, \dots, v_n\in V$.

Thus, we have the following isomorphisms of homotopy operads:
\begin{align*}
	\mathbf{Hom}\big({\NjAoperad^\ac}, \End_V\big)
	\cong&\ \mathbf{Hom}\big({\NjAoperad^\ac}, \mathbf{Hom}(\cals, \End_{sV})\big)\\
	\cong&\ \mathbf{Hom}\big({\NjAoperad^\ac}\ot_{\mathrm{H}}\cals, \End_{sV}\big)\\
	=&\ \mathbf{Hom}\big({\mathscr{S}({\NjAoperad^\ac})}, \End_{sV}\big).
\end{align*}
We obtain $$\frakC_{\NjA}(V)\cong \mathbf{Hom}\big({\mathscr{S}({\NjAoperad^\ac})}, \End_{sV}\big)^{\prod}.$$ Recall that ${\mathscr{S}({\NjAoperad^\ac})}(n)=\bfk u_n\oplus \bfk v_n$ with $|u_n|=0$ and $|v_n|=1$. By definition $$\mathbf{Hom}\big({\mathscr{S}({\NjAoperad^\ac})}, \mathrm{End}_{sV}\big)(n)=\Hom(\bfk u_n\oplus \bfk v_n, \Hom((sV)^{\ot n}, sV)).$$
Each $f\in \Hom((sV)^{\ot n}, sV)$ and $g\in \Hom((sV)^{\ot n}, V)$ determine bijectively a map, respectively,
$\widetilde{f} \in \Hom(\bfk u_n, \Hom((sV)^{\ot n}, sV)) $ and $ \widehat{g} \in \Hom(\bfk v_n, \Hom((sV)^{\ot n}, sV))$
such that $\widetilde{f}(u_n)=f$ and $\widehat{g}(v_n)=(-1)^{|g|}sg$.

Denote
$$\frakC_{\Alg}(V) = \Hom(\overline{T^c} (sV), sV) \quad \mathrm{and} \quad \frakC_{\NjO}(V) = \Hom(\overline{T^c} (sV), V),$$
where
$$\overline{T^c} (sV) = sV\oplus (sV)^{\ot 2}\oplus \cdots \oplus (sV)^{\ot n} \oplus \cdots $$
is the reduced cofree conilpotent tensor coalgebra generated by $sV$, with the usual deconcatenation coproduct.
In this way, we identify $\frakC_{\NjA}(V)$ with $\frakC_{\Alg}(V)\oplus \frakC_{\NjO}(V)$.

Notice that for a Nijenhuis associative algebra $A=(A, m, P)$, considered as a graded space concentrated in degree zero, $\frakC_{\NjA}(A)$ is just the underlying space of the cochain complex of Nijenhuis associative algebra $A$ up to shift.

A direct computation gives the $L_\infty$-algebra structure $\{l_n\}_{n\geqslant 1}$ on $\frakC_{\NjA}(V)$ as follows:
\begin{enumerate}
	
	\item For homogeneous elements $sf, sh\in \frakC_{\Alg}(V)$, define $$l_2(sf\ot sh):= [sf, sh]_{\G}\in\frakC_{\Alg}(V), $$ where the operation $[-, -]_{\G}$ is the Gerstenhaber bracket (see for instance, \cite{Ger63} or \cite[Page 9, (3)]{CGWZ24}).
	
	\item \label{Item L infty l-n}
	Let $n\geqslant 1$. For homogeneous elements $sh\in \Hom((sV)^{\ot n}, sV)\subset \frakC_{\Alg}(V)$ and $g_1, \dots, g_n\in \frakC_{\NjO}(V)$, 	define $$l_{n+1}(sh\ot g_1\ot \cdots \ot g_n)\in \frakC_{\NjO}(V)$$ as :
	\begin{align*}
		l_{n+1}&(sh\ot g_1\ot \cdots \ot g_n) :=\\
		& \sum_{\sigma\in \mathbb{S}_n}(-1)^{\eta} \Big(\sum_{k=0}^n (-1)^{\xi} s^{-1} \circ (sg_{\sigma (1)})
		\{ \cdots
		\{
		sg_{\sigma(k)}
		\{sh
		\{sg_{\sigma(k+1)}, \dots, sg_{\sigma(n)}\}
		\}
		\}
		\cdots \}
		\Big),
	\end{align*}
	where $(-1)^{\eta}=\chi(\sigma; g_1, \dots, g_n) (-1)^{ n(|h|+1)+ \sum\limits_{k=1}^{n-1} \sum\limits_{j=1}^{k} |g_{\sigma(j)}|} $, $ (-1)^{\xi} = (-1)^{(|h|+1) ( \sum\limits_{i=1}^{k}(|g_{\sigma(i)}|+1))+k}$
	and the operation $ -\{-\} $ is the brace operation introduced in \cite{GV95} (see also \cite[Definition 3.6]{CGWZ24}).
		
	\item Let $n\geqslant 1$. For homogeneous elements $sh\in \Hom((sV)^{\ot n}, sV)\subset \frakC_{\Alg}(V)$ and $g_1, \dots, g_n\in \frakC_{\NjO}(V)$, for $1\leqslant k\leqslant n$, define $$l_{n+1}(g_1\ot \cdots\ot g_k\ot sh \ot g_{k+1}\ot \cdots\ot g_n)\in \frakC_{\NjO}(V)$$ to be
	$$l_{n+1}(g_1\ot \cdots\ot g_k\ot sh \ot g_{k+1}\ot \cdots\ot g_n):=(-1)^{(|h|+1)(\sum\limits_{j=1}^k|g_j|)+k}l_{n+1}(sh\ot g_1\ot \cdots \ot g_n), $$
	where the RHS has been introduced in \eqref{Item L infty l-n}.
		
	\item All other components of operators $\{l_n\}_{n \geqslant 1}$ vanish.
\end{enumerate}



For later use, with this $L_\infty$-algebra, we can use the Maurer-Cartan method to redefine homotopy Nijenhuis associative algebras, which is equivalent to Definition~\ref{defn. homotopy Nijenhuis associative algebra via dg operad} according to \cite[Proposition 3.9]{CGWZ24}, as follows:

\begin{defn} \label{Def: homotopy NjA algebras}
	Let $V$ be a graded space. A \textit{homotopy Nijenhuis associative algebra structure} (or a $\NjAoperad_\infty$-algebra structure) on $V$ is defined to be a Maurer-Cartan element in the $L_\infty$-algebra ${\frakC_{\NjA}}(V)$.
\end{defn}

Let's make the definition explicit. Given an element $\alpha=(\{b_i\}_{i\geqslant 1}, \{R_i\}_{i\geqslant 1}) \in \frakC_{\NjA}(V)_{-1}$
with $b_i:(sV)^{\ot i}\rightarrow sV$ and $R_i:(sV)^{\ot i}\rightarrow V$, then $\alpha$ satisfies the Maurer-Cartan equation if and only if for each $n\geqslant 1$, the following equations hold:
\begin{eqnarray} \label{Eq: A-infinity NjA}
	\sum_{i=1}^{n} b_{n-i+1}\{b_{i}\}
	=0,
\end{eqnarray}
\begin{eqnarray} \label{Eq: homotopy-NjA-operator NjA}
	\sum_{\substack{ r_1+\cdots+r_p =n \\
		r_1, \dots, r_p \geqslant 1 \\
		1 \leqslant p \leqslant n }} \sum_{ t=0}^{p} (-1)^{t}
	(sR_{r_{1}})
	\{ \cdots
	\{
	sR_{r_{t}}
	\{b_{p}
	\{sR_{r_{t+1}}, \dots, sR_{r_{p}}\}
	\}
	\}
	\cdots \} = 0.
\end{eqnarray}

\begin{remark} \label{Rem. NiA[1]-algebra}
	We also say that $(sV, \{b_i\}_{i\geqslant 1}, \{R_i\}_{i\geqslant 1})$ forms a \textit{$\NjAoperad_\infty[1]$-algebra}, that is, a shifted homotopy Nijenhuis associative algebra.
\end{remark}

For later use, we need to fix two isomorphisms:
\begin{eqnarray} \label{Eq: first can isom}
	\Hom((sV)^{\ot n}, sV) \simeq \Hom(V^{\ot n}, V), \quad f \mapsto \widetilde{f}:= s^{-1} \circ f \circ s^{\ot n}
\end{eqnarray}
for $f \in \Hom((sV)^{\ot n}, sV)$, and
\begin{eqnarray} \label{Eq: second can isom}
	\Hom((sV)^{\ot n}, V) \simeq \Hom(V^{\ot n}, V), \quad g \mapsto \widehat{g}:= g \circ s^{\ot n}
\end{eqnarray}
for $g \in \Hom((sV)^{\ot n}, V)$.
These two fixed isomorphisms give the equivalences between Equation~\eqref{Eq: A-infinity NjA} and Equation~\eqref{Eq: stasheff-id NjA}, and between Equation~\eqref{Eq: homotopy-NjA-operator NjA} and Equation~\eqref{Eq: homotopy NjA-operator-version-2 NjA}, when we denote
$$ m_n = \widetilde{b_{n}} = s^{-1}\circ b_n\circ s^{\ot n}: V^{\ot n}\rightarrow V
\quad \mathrm{and} \quad
P_n = \widehat{R_{n}} = R_n\circ s^{\ot n}: V^{\ot n}\rightarrow V.$$



\smallskip

\subsection{Realising Nijenhuis associative algebra structures as Maurer-Cartan elements}\

In this subsection, we will see that the cohomology of Nijenhuis associative algebras introduced in Subsection \ref{Subsec: cohomology NjA} can be obtained using the twisting procedure in the $L_\infty$-algebra introduced in previous subsection.

The following result follows from \cite[Proposition 3.9]{CGWZ24}.


\begin{prop}\label{Prop: NjA is MC element}
	Let $V$ be an ungraded space considered as a graded space concentrated in degree 0. Then a Nijenhuis associative algebra structure on $V$ is equivalent to a Maurer-Cartan element in the $L_\infty$-algebra ${\frakC_{\NjA}}(V)$.
\end{prop}

Let us justify the cohomology theory of Nijenhuis associative algebras introduced in Section~\ref{Sect: Cohomology theory of Nijenhuis associative algebras}.

\begin{prop} \label{Prop: cohomology complex as the underlying complex of L-infinity algebra NjA}
	Let $A = (A, m, P)$ be a Nijenhuis associative algebra. Twist the $L_\infty$-algebra $\frakC_{\NjA}(A)$ by the Maurer-Cartan element corresponding to the Nijenhuis associative algebra structure, then its underlying complex is isomorphic to $s\C^\bullet_{\NjA}(A)$, the shift of the cochain complex of Nijenhuis associative algebra $A$ defined in Subsection \ref{Subsec: cohomology NjA}.
\end{prop}

\begin{proof}
	By Proposition~\ref{Prop: NjA is MC element}, the Nijenhuis associative algebra structure on $A$ is equivalent to a Maurer-Cartan element $\alpha=(\nu, \tau)$ in the $L_\infty$-algebra $\frakC_{\NjA}(A)$ with
	$$\nu=-s\circ m \circ (s^{-1})^{\ot 2}: (sA)^{\ot 2}\rightarrow sA \quad \mathrm{and} \quad \tau=P\circ s^{-1}: sA\rightarrow A.$$
	Via twisting procedure (see \cite[Proposition 2.4]{CGWZ24}), this Maurer-Cartan element induces a new $L_\infty$-algebra structure $\{l_n^\alpha\}_{n\geqslant 1}$ on the graded space $\frakC_{\NjA}(A)$.
	More precisely, for any $(sf, g)\in\Hom((sA)^{\ot n}, sA) \oplus \Hom((sA)^{\ot n-1}, A) \subset \frakC_{\NjA}(A)$
	, we have
	\begin{align*}
		l_1^\alpha(sf)
		&=l_{1}(sf) -l_2(\alpha \ot sf) + \sum_{i=2}^n (-1)^{\frac{i(i+1)}{2}} \frac{1}{i!} l_{i+1}(\alpha^{\ot i}\ot sf)\\
		&=-l_2(\nu \ot sf) -l_2(\tau \ot sf) + \sum_{i=2}^n (-1)^{\frac{i(i+1)}{2}} \frac{1}{i!} l_{i+1}(\tau^{\ot i}\ot sf)\\
		&=-[\nu, sf]_{\G} + (-1)^{\frac{n(n+1)}{2}} \frac{1}{n!} l_{n+1}(\tau^{\ot n}\ot sf).
	\end{align*}
	Note that, $-[\nu, sf]_{\G}$ corresponds to $-(-1)^{n+1}\delta_{\Alg}^n(\widetilde{sf})$ in $s\C^\bullet_\Alg(A)$
	under the fixed isomorphism \eqref{Eq: first can isom}.
	And
	\begin{align*}
	 (-1)^{\frac{n(n+1)}{2}} \frac{1}{n!} l_{n+1}(\tau^{\ot n}\ot sf)
		&= (-1)^{\frac{n(n+1)}{2}} (-1)^{n(|f|+1)+n} \frac{1}{n!} l_{n+1}( sf \ot \tau^{\ot n})\\
		&= \sum_{k=0}^{n} (-1)^{k} s^{-1} \underbrace{(s\tau) \circ \cdots \circ (s\tau)}_{k} \circ sf \{\underbrace{s\tau, \dots, s\tau}_{n-k}\},
	\end{align*}
	which corresponds to $\Phi^n(\widetilde{sf})$ under the fixed isomorphism \eqref{Eq: second can isom}.
We also have
	\begin{align*}
		l_1^{\alpha}(g)
		=&\ l_1(g) - l_2(\alpha \ot g) - \frac{1}{2!} l_3(\alpha \ot \alpha \ot g)\\
		=&\ -l_3(\nu \ot \tau\ot g)\\
		=&\ s^{-1}\nu \{s\tau, sg\} - \tau\{\nu\{sg\}\} + (-1)^{n} \tau \{sg\{\nu\}\} \\
		&+ s^{-1}\nu \{sg, s\tau\} - (-1)^{n}g\{\nu\{s\tau\}\} + (-1)^{n} g \{s\tau\{\nu\}\},
	\end{align*}
	which corresponds to $ (-1)^{n}\delta_{\NjO}^{n-1}(\widehat{g})$ under the fixed isomorphism \eqref{Eq: second can isom}.
	
	In conclusion, 	we have that the underlying complex of the twisted $L_\infty$-algebra $\frakC_{\NjA}(A)$ by Maurer-Cartan element $\alpha$ is isomorphic to the cochain complex $s\C^\bullet_{\NjA}(A)$.
\end{proof}

Although $\frakC_{\NjA}(A)$ is an $L_\infty$-algebra, the next proposition shows that once the associative algebra structure $m$ on $A$ is fixed, the graded space $\frakC_{\NjO}(A)$, which, after twisting procedure, controls deformations of Nijenhuis operators, is a genuine graded Lie algebra.

\begin{prop} \label{Prop: justifying Cohomology theory of Nijenhuis operator Ass}
	Let $(A, m)$ be an associative algebra. Then we have the following statements:
	\begin{enumerate}
		\item The graded space $\frakC_{\NjO}(A)$ can be endowed with a graded Lie algebra structure, and a Nijenhuis operator on $(A, m)$ is equivalent to a Maurer-Cartan element in this graded Lie algebra.
		
		\item Given a Nijenhuis operator $P$ on associative algebra $(A, m)$, the underlying complex of the twisted dg Lie algebra $\frakC_{\NjO}(A)$ by the corresponding Maurer-Cartan element is isomorphic to the cochain complex $\C_{\NjO}^\bullet(A)$ of Nijenhuis operator $P$ defined in Subsection~\ref{Subsect: cohomology NjA operator}.
	\end{enumerate}
\end{prop}

\begin{proof}
	(i) Consider $A$ as a graded space concentrated in degree 0. Define
	$$ \nu =- s\circ m \circ (s^{-1})^{\ot 2}: (sA)^{\ot 2}\rightarrow sA.$$
	Note that $\alpha=(\nu, 0)$ is naturally a Maurer-Cartan element in $L_\infty$-algebra $\frakC_{\NjA}(A)$.
	By the construction of $\{l_n\}_{n\geqslant 1}$ on $\frakC_{\NjA}(A)$, the graded subspace $\frakC_{\NjO}(A)$ is closed under the action of $\{l_n^\alpha\}_{n\geqslant 1}$. Note that $ l_1^\alpha \equiv 0 $ on $ \frakC_{\NjO}(A) $, and the restriction of $l_n^\alpha$ on $\frakC_{\NjO}(A)$ is $0$ for $n\geqslant 3$ because the arity of $\nu$ is 2. Thus $(\frakC_{\NjO}(A), \{l_2^\alpha\})$ forms a graded Lie algebra.
	More explicitly, for any $f\in \Hom((sA)^{\ot n}, A)$, $g\in \Hom((sA)^{\ot k}, A)$, we have
	\begin{align*}
		l_2^\alpha(f\ot g)
		=&\ l_{2}(f, g)+l_{3}(\alpha, f, g)\\
		=&\ l_{3}(\nu, f, g)\\
		=&\ (-1)^{|f|}\big(s^{-1}\nu\{sf, sg\}-(-1)^{|f|+1} f\{\nu \{sg\} \} + (-1)^{|f|+1+|g|+1}f\{sg\{\nu\}\}\big)\\
		&+(-1)^{|f||g|+1+|g|}\big(s^{-1}\nu\{sg, sf\}-(-1)^{|g|+1} g\{\nu \{sf\} \}+ (-1)^{|g|+1+|f|+1}g\{sf\{\nu\}\}\big)\\
		=&\ (-1)^{n}s^{-1}\nu\{sf, sg\} + f\{\nu \{sg\}\} + (-1)^{k}f\{sg\{\nu\}\}\\
		&+(-1)^{nk+k+1}s^{-1}\nu\{sg, sf\}-(-1)^{nk} g\{\nu \{sf\} \}+ (-1)^{nk+n+1}g\{sf\{\nu\}\}.
	\end{align*}
	
	Since $A$ is concentrated in degree 0, we have $\frakC_{\NjO}(A)_{-1}=\Hom(sA, A)$. Take an element $\tau \in \Hom(sA, A)$, then $\tau$ satisfies the Maurer-Cartan equation
	$$-\frac{1}{2}l_2^{\alpha}(\tau\ot \tau)=0$$
	if and only if
	$$s^{-1}\nu\{s\tau, s\tau\}-\tau\{\nu\{s\tau\}\}+\tau\{s\tau\{\nu\}\}=0.$$
	Define $P=\tau\circ s:A\rightarrow A$. The above equation can be rewritten as
	$$ m \circ (P \ot P) = P \circ m \circ (P \ot \id) + P \circ m \circ (\id \ot P) - P \circ P \circ m, $$
	which means that $P$ is a Nijenhuis operator on associative algebra $(A, m)$.

	(ii) Now let $P$ be a Nijenhuis operator on associative algebra $(A, m)$. By the first statement, $P$ corresponds to a Maurer-Cartan element $\beta = P \circ s^{-1}$ in the graded Lie algebra $(\frakC_{\NjA}(A), l_2^\alpha)$.
	For $f\in \Hom((sA)^{\ot n}, A) \subset \frakC_{\NjO}(A)$, we compute $(l_1^\alpha)^\beta(f)$ as follows:
	$$\begin{aligned}
		(l_1^\alpha)^\beta(f)
		=&\ l_1^\alpha(f)-l_2^\alpha(\beta\ot f)\\
		=&\ -l_2^\alpha(\beta\ot f)\\
		=&\ s^{-1}\nu\{s\beta, sf\} - \beta\{\nu\{sf\}\} - (-1)^{n} \beta \{sf\{\nu\}\}\\
		&+ s^{-1}\nu\{sf, s\beta\} +(-1)^{n}f\{\nu\{s\beta\}\} - (-1)^{n} f \{s\beta\{\nu\}\},
	\end{aligned}$$
	which corresponds to $ (-1)^{n+1}\delta_{\NjO}^{n}(\widehat{f})$ under the fixed isomorphism \eqref{Eq: second can isom}. So the underlying complex of the twisted dg Lie algebra $\frakC_{\NjO}(A)$ by the corresponding Maurer-Cartan element $\beta$ is isomorphic to the cochain complex $\C_{\NjO}^\bullet(A)$ of Nijenhuis operator $P$.	
\end{proof}

\medskip

\section{Homotopy relative Rota-Baxter associative algebras and homotopy Nijenhuis associative algebras} \label{Section: relation between HrelRBA and HNjA}

In this section, we will establish a connection between homotopy relative Rota-Baxter associative algebras of weight $0$ and homotopy Nijenhuis associative algebras. This generalizes the following result of Das \cite{Das2020}, which relates Rota-Baxter associative algebras of weight $0$ and Nijenhuis associative algebras, to the homotopy level:

\begin{prop} \cite[2.5 Proposition]{Das2020}
	Let $(A, m)$ be an associative algebra and $(M, l, r)$ be a bimodule over it, where $l: A \ot M \rightarrow M$ (resp. $r: M \ot A \rightarrow M$) denotes the left (resp. right) action of $A$ on $M$. A linear map $P: M \rightarrow A$ is a \textit{relative Rota-Baxter operator of weight $0$} on $A$ with respect to $M$, i.e.,
	$$ m \circ (P \ot P) = P \circ ( l \circ (P \ot \id) + r \circ(\id \ot P)), $$
	if and only if $N_P : A \oplus M \rightarrow A \oplus M, (a, x) \mapsto (P(x), 0)$ is a Nijenhuis operator on the semi-direct product associative algebra $A \ltimes M = (A \oplus M, m+l+r)$.
\end{prop}

For a graded space $V$, it is well known that $(V, \{m_{n}\}_{n \geqslant 1})$ being an $A_{\infty}$-algebra is equivalent to its suspension $(sV, \{\frakm_{n}\}_{n \geqslant 1})$ being an $A_{\infty}[1]$-algebra, where the correspondence is given by $m_{n} = s^{-1} \circ \frakm_{n} \circ s^{\ot n}$ as described in Equation~\eqref{Eq: first can isom}.
The concepts of left modules, right modules and bimodules over an $A_\infty$-algebra were introduced by Burke \cite{Burke}, Keller \cite{Keller01, Keller06} and Tradler \cite{Tradler2008}.
Das and Mishra \cite{DasMis22} gave the corresponding concepts in the framework of $A_\infty[1]$-algebras.
Using the notion of brace operation, the concepts of $A_\infty[1]$-algebras and $A_\infty[1]$-bimodules can be equivalently defined as follows.


\begin{defn} \label{Def: A infinity[1] algebras}
	Let $\frakA$ be a graded space. It is called an \textit{$A_\infty[1]$-algebra} if $\frakA$ is endowed with a family of graded linear operators $\frakm_k:\frakA^{\ot k}\rightarrow \frakA, k\geqslant 1$, with $|\frakm_k|=-1$, subjecting to the following identities: for $n \geq 1$,
	\begin{align}\label{Eq: A infty[1] algebras}
		\sum\limits_{i+j=n+1} \frakm_{j} \{\frakm_{i}\} =0,
	\end{align}
	where $ -\{-\} $ is the brace operation.
\end{defn}

\begin{defn}\label{Def: A-infty[1] module}
	Let $(\frakA, \{\frakm_{n}\}_{n \geqslant 1})$ be an $A_\infty[1]$-algebra and $\frakM$ be a dg space with differential $\varrho_{1}$ of degree $-1$. An \textit{$\frakA$-bimodule structure}
	on $\frakM$ is a family of graded linear operators $\varrho_k = \sum\limits_{p=1}^{k} \varrho_{k}^{p}: (\frakA \oplus \frakM)^{\ot k} \rightarrow \frakM$, $k\geqslant 2$, where $ \varrho_{k}^{p} : \frakA \ot \cdots \ot \frakA \ot \underset{p-th}{\frakM} \ot \frakA \ot \cdots \ot \frakA \rightarrow \frakM$, $1 \leqslant p \leqslant k$, with $|\varrho_{k}^{p}|=-1$, subjecting to: for $n \geq 1$,
	\begin{align}\label{Eq: A infty [1] module}
		\sum\limits_{i+j=n+1} \varrho_{j} \{\frakm_{i}\} + \varrho_{j} \{\varrho_{i}\}=0.
	\end{align}	
\end{defn}

Inspired by the construction of semi-direct product of an associative algebra by a bimodule over it, we have the following general construction:
\begin{prop} \label{Prop: semi-direct sum A-infty[1]}
	Let $(\frakA, \{\frakm_{n}\}_{n \geqslant 1})$ be an $A_{\infty}[1]$-algebra and $(\frakM, \{\varrho_{n}\}_{n \geqslant 1})$ be an $\frakA$-bimodule. Then the graded space $\frakA \oplus \frakM $ has a natural $A_{\infty}[1]$-algebra structure given by
	$$\frakh_{k}:= \frakm_{k} + \varrho_{k} : (\frakA \oplus \frakM)^{\ot k} \rightarrow \frakA \oplus \frakM,\ k\geqslant 1.$$
	This $A_{\infty}[1]$-algebra is called the \textit{semi-direct product} of $\frakA$ and $\frakM$, denoted by $\frakA \ltimes \frakM$.
\end{prop}

\begin{proof}
	Note that, for any $i, j \geqslant 1$, $\frakm_{j}\{\varrho_{i}\}$ is always equal to $0$ because the output of $\varrho_{i}$ falls in space $\frakM$ however any input of $\frakm_{j}$ falls in space $\frakA$.
	For any $n \geqslant 1$, according to Equations~\eqref{Eq: A infty[1] algebras}-\eqref{Eq: A infty [1] module}, we have
	\begin{align*}
		\sum\limits_{i+j=n+1} \frakh_{j} \{\frakh_{i}\}
		= \sum\limits_{i+j=n+1} (\frakm_{j} \{\frakm_{i}\} + \varrho_{j} \{\frakm_{i}\} + \varrho_{j} \{\varrho_{i}\})
		= 0.
	\end{align*}
	This means that $\{\frakh_{n}\}_{n \geqslant 1}$ is indeed an $A_{\infty}[1]$-algebra structure on the space $\frakA \oplus \frakM$.
\end{proof}

%

{We now recall the definition of homotopy relative Rota-Baxter associative algebras of weight $0$. As $A_\infty$-algebras generalize associative algebras, and homotopy Nijenhuis associative algebras generalize Nijenhuis associative algebras, homotopy relative Rota-Baxter associative algebras of weight $0$ represent the homotopy version of Rota-Baxter associative algebras of weight $0$.
}

\begin{defn} \cite{DasMis22,Song24} \label{defn. homotopy relative Rota-Baxter associative algebras}
	Let $A$ and $M$ be two graded spaces. Assume that $(sA, \{\frakm_{n}\}_{n \geqslant 1})$ is an $A_{\infty}[1]$-algebra
	and $(sM, \{\varrho_{n}\}_{n \geqslant 1})$ is an $sA$-bimodule.
	A degree $-1$ element $B = \sum\limits_{i \geqslant 1} B_{i} \in \Hom(\overline{T^c}(sM), A)$ with $B_{i}:(sM)^{\ot i} \rightarrow A$ is called a \textit{homotopy relative Rota-Baxter operator of weight $0$ on $A$ with respect to $M$} if the following equation holds for all $n \geqslant 1$:
	\begin{eqnarray} \label{Eq: homotopy-rl-rb-operator Ass}
		\sum_{r_1+\cdots+r_p =n \atop
			r_1, \dots, r_p \geqslant 1, p \geqslant 1}
		\frakm_{p}\{sB_{r_{1}}, \dots, sB_{r_{p}}\} 	
		= \sum_{r_1+\cdots+r_p =n \atop
			r_1, \dots, r_p \geqslant 1, p \geqslant 1}
		(sB_{r_{1}})\{\varrho_{p}\{sB_{r_{2}}, \dots, sB_{r_{p}}\}
		\}.
	\end{eqnarray}
	In this case, $(A, \{\frakm_{n}\}_{n \geqslant 1}, M, \{\varrho_{n}\}_{n \geqslant 1}, B)$ is called a \textit{homotopy relative Rota-Baxter associative algebra of weight $0$}.
\end{defn}
\begin{remark}By explicitly expanding Equation~\eqref{Eq: homotopy-rl-rb-operator Ass} for small $n$, it can be observed that the operator $B_1$ serves as a relative Rota-Baxter operator up to homotopy with respect to $(\frak{m}_2, \varrho_2^1,\varrho_2^2)$, with the homotopy provided by the operator $B_2$.
\end{remark}

Recall that a homotopy Nijenhuis associative algebra structure on a graded space $V$ is defined to be a Maurer-Cartan element in the $L_{\infty}$-algebra $\frakC_{\NjA}(V)$
as introduced in Definition~\ref{Def: homotopy NjA algebras}. More explicitly, a Maurer-Cartan element in $\frakC_{\NjA}(V)$
is an element $$\alpha=(\{b_i\}_{i\geqslant 1}, \{R_i\}_{i\geqslant 1})\in \frakC_{\NjA}(V)_{-1}
=\Hom(\overline{T^c}(sV), sV)_{-1}\oplus \Hom(\overline{T^c}(sV), V)_{-1} $$
with $b_i:(sV)^{\ot i}\rightarrow sV$ and $R_i:(sV)^{\ot i}\rightarrow V$ which satisfy Equations~\eqref{Eq: A-infinity NjA}-\eqref{Eq: homotopy-NjA-operator NjA}.

We establish a connection between homotopy relative Rota-Baxter associative algebras of weight $0$ and homotopy Nijenhuis associative algebras as follows:

\begin{thm} \label{Thm: RBA and NjA}
	Let $A$ and $M$ be two graded spaces.
	Assume that $(sA, \{\frakm_{n}\}_{n \geqslant 1})$ is an $A_{\infty}[1]$-algebra and $(sM, \{\varrho_{n}\}_{n \geqslant 1})$ is an $sA$-bimodule. A degree $-1$ element $B = \sum\limits_{i \geqslant 1} B_{i} \in \Hom(\overline{T^c}(sM), A)$ with $B_{i}:(sM)^{\ot i} \rightarrow A$ is a homotopy relative Rota-Baxter operator of weight $0$ on $A$ with respect to $M$ if and only if
	$A \oplus M$ is a homotopy Nijenhuis associative algebra corresponding to the Maurer-Cartan element $(\{b_n\}_{n\geqslant 1}, \{R_n\}_{n\geqslant 1})$ with $b_{n} = \frakm_{n} + \varrho_{n}$, $R_{n} = B_{n}$, $n \geqslant 1$.
\end{thm}

\begin{proof}
	Let $B$ be a homotopy relative Rota-Baxter operator of weight $0$ on $A$ with respect to $M$. To prove that $(\{b_n\}_{n\geqslant 1}, \{R_n\}_{n\geqslant 1})$ is a Maurer-Cartan element in the $L_{\infty}$-algebra $\frakC_{\NjA}(A \oplus M)$,
	 one just needs to check that the following identities hold:
	\begin{eqnarray}
			\label{Eq: A-infinity NjA check 1}
			\sum_{i=1}^{n} b_{n-i+1}\{b_{i}\} = 0,
		\end{eqnarray}
	\begin{eqnarray}
			\label{Eq: homotopy-NjA-operator NjA check 2}
			\sum_{ r_1+\cdots+r_p =n \atop
					r_1, \dots, r_p \geqslant 1, p \geqslant 1 } \sum_{ t=0}^{p} (-1)^{t}
			(sR_{r_{1}})
			\{ \cdots
			\{
			sR_{r_{t}}
			\{b_{p}
			\{sR_{r_{t+1}}, \dots, sR_{r_{p}}\}
			\}
			\}
			\cdots \} = 0.
		\end{eqnarray}
	On the one hand, Equation~\eqref{Eq: A-infinity NjA check 1} holds since $\{b_{n}\}_{n \geqslant 1}$ is exactly the $A_{\infty}[1]$-algebra structure of the semi-direct product $sA \ltimes sM$ as stated in Proposition~\ref{Prop: semi-direct sum A-infty[1]}. On the other hand, Equation~\eqref{Eq: homotopy-rl-rb-operator Ass} implies
	\begin{align*}
			&\ \sum_{ r_1+\cdots+r_p =n \atop
					r_1, \dots, r_p \geqslant 1, p \geqslant 1 }
			\sum_{ t=0}^{p} (-1)^{t}
			(sR_{r_{1}})\{ \cdots\{sR_{r_{t}}
			\{b_{p}
			\{sR_{r_{t+1}}, \dots, sR_{r_{p}}\}
			\}\}\cdots \} \\
			= &\
			\sum_{ r_1+\cdots+r_p =n \atop
					r_1, \dots, r_p \geqslant 1, p \geqslant 1 }
			\sum_{ t=0}^{p} (-1)^{t}
			(sB_{r_{1}})\{ \cdots\{sB_{r_{t}}
			\{\frakm_{p}
			\{sB_{r_{t+1}}, \dots, sB_{r_{p}}\}
			\}\}\cdots \}\\
			&\ + \sum_{ r_1+\cdots+r_p =n \atop
					r_1, \dots, r_p \geqslant 1, p \geqslant 1 }
			\sum_{ t=0}^{p} (-1)^{t}
			(sB_{r_{1}})\{ \cdots\{sB_{r_{t}}
			\{\varrho_{p}
			\{sB_{r_{t+1}}, \dots, sB_{r_{p}}\}
			\}\}\cdots \}\\
			= &\
			\sum_{ r_1+\cdots+r_p =n \atop
					r_1, \dots, r_p \geqslant 1, p \geqslant 1 } 
			\frakm_{p}
			\{sB_{r_{1}}, \dots, sB_{r_{p}}\}
			+ \sum_{ r_1+\cdots+r_p =n \atop
					r_1, \dots, r_p \geqslant 1, p \geqslant 1 } 
			(sB_{r_{1}})
			\{\varrho_{p}
			\{sB_{r_{2}}, \dots, sB_{r_{p}}\}\}	\\
			= &\ \ 0.
		\end{align*}
	So Equation~\eqref{Eq: homotopy-NjA-operator NjA check 2} also holds. Thus $(\{b_n\}_{n\geqslant 1}, \{R_n\}_{n\geqslant 1})$ defines a homotopy Nijenhuis associative algebra structure on $A\oplus M$.
	
	The proof in the opposite direction is obvious.
\end{proof}

Denote by $\mathfrak{RBA}^{rel}$ the coloured operad for relative Rota-Baxter associative algebras of weight $0$, and by $\mathfrak{RBA}^{rel}_{\infty}$ the coloured dg operad for homotopy relative Rota-Baxter associative algebras of weight $0$.
Here is a direct corollary of Theorem~\ref{Thm: RBA and NjA}.

\begin{cor} \label{Cor: from relRBA infty to NiA infty}
	View $\mathfrak{RBA}^{rel}_{\infty}$ as a dg operad by forgetting its colours, then there is a dg operad morphism from $\NjAoperad_{\infty}$ to $\mathfrak{RBA}^{rel}_{\infty}$ which lifts the operad morphism from $\NjAoperad$ to $\mathfrak{RBA}^{rel}$, i.e., we have the following commutative diagram of dg operads:
	\[\xymatrix{\NjAoperad_{\infty}\ar[r]\ar[d]&\mathfrak{RBA}^{rel}_\infty \ar[d]\\
\NjAoperad\ar[r]&\mathfrak{RBA}^{rel}.	
	}\]
\end{cor}

\medskip

\textbf{Acknowledgements} 

The first, second, and fourth authors were supported by the National Natural Science Foundation of China (No. 12071137), the fourth author was also financed by Key Laboratory of MEA (Ministry of Education), by Shanghai Key Laboratory of PMMP (No. 22DZ2229014).
The third author was supported by the National Natural Science Foundation of China (No. 12101183), and also supported by the Postdoctoral Fellowship Program of CPSF under Grant Number (GZC20240406). 


\textbf{Conflict of Interest}

None of the authors has any conflict of interest in the conceptualization or publication of this
work.

\textbf{Data availability}

Data sharing is not applicable to this article as no new data were created or analyzed in this study.

\medskip

\end{document}